\documentclass[11pt]{amsart}
\usepackage{amsfonts,amssymb,amscd,amsmath,enumerate,verbatim,calc}

\usepackage{amsmath}
\usepackage{amssymb}
\usepackage{amscd}

\def\NZQ{\mathbb}               
\def\NN{{\NZQ N}}
\def\QQ{{\NZQ Q}}
\def\ZZ{{\NZQ Z}}
\def\RR{{\NZQ R}}

\def\PP{{\NZQ P}}

\newtheorem{Theorem}{Theorem}[section]
\newtheorem{Lemma}[Theorem]{Lemma}
\newtheorem{Corollary}[Theorem]{Corollary}
\newtheorem{Proposition}[Theorem]{Proposition}
\newtheorem{Remark}[Theorem]{Remark}

\newtheorem{Example}[Theorem]{Example}

\newtheorem{Definition}[Theorem]{Definition}

\let\epsilon\varepsilon
\let\phi=\varphi
\let\kappa=\varkappa

\def\mR{m_R}
\def\Ass{\operatorname{Ass}}
\def\mfp{\mathfrak p}
\def\mfa{\mathfrak a}
\def\Spec{\operatorname{Spec}}
\def \Min{\operatorname{Min}}

\begin{document}
\title{Multiplicities and Mixed Multiplicities  of arbitrary Filtrations}

\author{Steven Dale Cutkosky}
\author{Parangama Sarkar}

\thanks{The first author was partially supported by NSF grant DMS-1700046.}
\thanks{The second author was partially supported by the DST, India: INSPIRE Faculty Fellowship.}

\address{Steven Dale Cutkosky, Department of Mathematics,
University of Missouri, Columbia, MO 65211, USA}
\email{cutkoskys@missouri.edu}

\address{Parangama Sarkar,  Department of Mathematics,
Indian Institute of Technology, Palakkad, India}
\email{parangama@iitpkd.ac.in}

\begin{abstract} We develop a theory of multiplicities and mixed multiplicities of  filtrations, extending the theory for filtrations of $m$-primary ideals to arbitrary (not necessarily Noetherian) filtrations. The mixed multiplicities of $r$ filtrations on an analytically unramified local ring $R$ come from  the coefficients of a suitable homogeneous polynomial in $r$ variables of degree equal to the dimension of the ring, analogously to the classical case of the mixed multiplicities of $m$-primary ideals in a local ring. 
We prove that the Minkowski inequalities hold for arbitrary filtrations. The characterization of equality in the Minkowski inequality for m-primary ideals in a local ring by Teissier, Rees and Sharp and Katz does not extend to arbitrary filtrations, but we show that they are true in a large and important subcategory of filtrations. We define divisorial and bounded filtrations. The filtration of powers of a fixed ideal is a bounded filtration, as is a divisorial filtration.  We show that in an excellent local domain, the  characterization of equality in the Minkowski equality is characterized by  the condition that the  integral closures of suitable Rees like algebras are the same, strictly generalizing the theorem of Teissier, Rees and Sharp and Katz. We also prove that a theorem of Rees characterizing the inclusion  of ideals with the same multiplicity generalizes to bounded filtrations in excellent local domains.
We give a number of other applications, extending classical theorems for ideals. 
\end{abstract}


\maketitle

\section{Introduction} In this paper we extend the theory of multiplicities and mixed multiplicities of filtrations of $m_R$-primary ideals in a local ring $R$ to arbitrary filtrations. We prove that these multiplicities enjoy  good properties and derive  applications.

The study of mixed multiplicities of $m_R$-primary ideals in a  local ring $R$ with maximal ideal $m_R$  was initiated by Bhattacharya \cite{Bh}, Rees  \cite{R} and Teissier  and Risler \cite{T1}. 
In \cite{CSS}  the notion of mixed multiplicities is extended to arbitrary,  not necessarily Noetherian, filtrations of $R$ by $m_R$-primary ideals ($m_R$-filtrations). It is shown in \cite{CSS} that many basic theorems for mixed multiplicities of $m_R$-primary ideals are true for $m_R$-filtrations. 

The development of the subject of mixed multiplicities and its connection to Teissier's work on equisingularity \cite{T1} is explained in \cite{GGV}.   A  survey of the theory of  multiplicities and mixed multiplicities of  $m_R$-primary ideals  can be found in  \cite[Chapter 17]{HS}, including discussion of the results of  the papers \cite{R1} of Rees and \cite{S} of  Swanson, and the theory of Minkowski inequalities of Teissier \cite{T1}, \cite{T2}, Rees and Sharp \cite{RS} and Katz \cite{Ka}.   Later, Katz and Verma \cite{KV}, generalized mixed multiplicities to ideals that are not all $m_R$-primary.  Trung and Verma \cite{TV} computed mixed multiplicities of monomial ideals from mixed volumes of suitable polytopes.  Mixed multiplicities are also used by Huh in the analysis of the coefficients of the chromatic polynomial of graph theory in \cite{H}.

A notion of mixed multiplicity for arbitrary ideals is introduced by Bhattacharya in \cite{Bh}.  This notion of mixed multiplicity is extended to arbitrary graded families of ideals by Cid Ruiz and Monta\~no, \cite{CM}.
We give an alternate definition of multiplicity and mixed multiplicity  for filtrations of ideals  in this paper, which more strictly generalizes the definition for $m_R$-primary ideals.

All local rings will be assumed to be Noetherian.
Let $R$ be a $d$-dimensional local ring with maximal ideal $m_R$.
It is shown in \cite[Theorem 1.1]{C1} and \cite[Theorem 4.2]{C3}
  that in a local ring $R$,  the limit
  \begin{equation}\label{I5}
  \lim_{n\rightarrow\infty}\frac{\ell_R(R/I_n)}{n^d}
\end{equation}
exists for all filtrations $\mathcal I=\{I_n\}$ of $m_R$-primary ideals if and only if $\dim N(\hat R)<\dim R$, where $N(\hat R)$ is the nilradical of $R$. In local rings $R$ which satisfy $\dim N(\hat R)<\dim R$, we may then define the multiplicity of a filtration  $\mathcal I$ of $m_R$-primary ideals by
$$
e_{R}(\mathcal I)= \lim_{n\rightarrow\infty}\frac{\ell_R(R/I_n)}{n^d/d!}.
$$

The problem of existence of such limits (\ref{I5}) has been considered by Ein, Lazarsfeld and Smith \cite{ELS} and Musta\c{t}\u{a} \cite{Mus}.
When the ring $R$ is a domain and is essentially of finite type over an algebraically closed field $k$ with $R/m_R=k$, Lazarsfeld and Musta\c{t}\u{a} \cite{LM} showed that
the limit exists for all $m_R$-filtrations.  In \cite{C2}, Cutkosky proved it in the complete generality  stated above. Lazarsfeld and Musta\c{t}\u{a}  use  in \cite{LM} the method of counting asymptotic vector space dimensions of graded families using ``Okounkov bodies''. This method, which is reminiscent of the geometric methods used by Minkowski in number theory, was developed by Okounkov \cite{Ok}, Kaveh and Khovanskii \cite{KK} and Lazarsfeld and   Musta\c{t}\u{a}   \cite{LM}. We also use this method. 
The fact that $\dim N(R)=d$ implies there exists a filtration without a limit was observed by Dao and Smirnov.

 It is shown in  \cite{CSS} 
  that if $R$ is a local ring such that $\dim N(\hat R)<d$ and 
  $$
  \mathcal I(1)=\{I(1)_m\},\ldots,\mathcal I(r)=\{I(r)_m\}
  $$
   are filtrations of $m_R$-primary ideals, then the limit  
 $$
 P(n_1,\ldots,n_r)=\lim_{m\rightarrow\infty}\frac{\ell_R(R/I(1)_{mn_1}\cdots I(r)_{mn_r})}{m^d/d!}
 $$
 is a homogeneous real polynomial $P(n_1,\ldots,n_r)$ of degree $d$ for $n_1,\ldots,n_r\in \NN$.
 We may thus define the mixed multiplicities $e_R(\mathcal I(1)^{[d_1]},\ldots,\mathcal I(r)^{[d_r]})$
 of these filtrations from the coefficients of this polynomial, by the expansion
  $$
  P(n_1,\ldots,n_r)=\sum_{d_1+\cdots+d_r=d}\frac{d!}{d_1!\cdots d_r!}e_R(\mathcal I(1)^{[d_1]},\ldots,\mathcal I(r)^{[d_r]})
  n_1^{d_1}\cdots n_r^{d_r}.
 $$
 In \cite{CSS}, multiplicities $e_R(\mathcal I;N)$ and $e_R(\mathcal I(1)^{[d_1]},\ldots,\mathcal I(r)^{[d_r]};N)$ are defined for a  finitely generated $R$-module $N$ and filtrations of $m_R$-primary ideals. 
 
 We now extend these definitions to arbitrary filtrations. Let $R$ be an analytically unramified local ring. 
 Let $\mfa$ be an $m_R$-primary ideal and $\mathcal I=\{I_n\}$ be a filtration of ideals on $R$. We have that $V(I_n)=V(I_1)$ 
 for all $n$ (Lemma \ref{dim}) where 
 $$
 V(I)=\{\mfp\in \mbox{Spec}(R)\mid I\subset \mfp\}
 $$
  and we define
 $$
 s(\mathcal I)=\dim R/I_1.
 $$
 For $s\in \NN$ and a finitely generated $R$-module $N$ such that $\dim N\le s$ (\cite[Section 2 of Chapter V]{Se} and \cite[Section 4.7]{BH}) 
 $$
 e_s(\mfa,N)=\left\{\begin{array}{ll}
 e_{\mfa}(N)&\mbox{ if }\dim N=s\\
 0 &\mbox{ if }\dim N<s.
 \end{array}\right.
 $$
 In Proposition \ref{Prop2}, we show that the limit
 $$
 \lim_{m\rightarrow\infty} \frac{e_s(\mfa,R/I_m)}{m^{d-s}/(d-s)!}
 $$
  exists for $s\ge s(\mathcal I)$ if $R$ is analytically unramified, and define the multiplicity $e_s(\mfa, \mathcal I)$ to be equal to this limit. We have an associativity formula (\ref{M10*}),
  $$
  e_s(\mfa,\mathcal I)=\sum_{\mfp}e_{R_{\mfp}}(\mathcal I_{\mfp})e_{\mfa}(R/\mfp),
  $$
  where $\mathcal I_{\mfp}=\{(I_n)_{\mfp}\}$ and  the sum is over all $\mfp\in \mbox{Spec}(R)$ such that $\dim R/\mfp=s$ and $\dim R_{\mfp}=d-s$.
 
The condition $R$ analytically unramified is used to ensure that the limits $e_{R_{\mfp}}(\mathcal I_{\mfp})$ exist  all for prime ideals  $\mfp$ in $R$. 

We then define mixed multiplicities of arbitrary filtrations $\mathcal I(1)=\{I(1)_m\},\ldots,\mathcal I(r)=\{I(r)_m\}$. We show in Theorem \ref{Theorem3} that for $s\ge \max\{s(\mathcal I(1)),\ldots,s(\mathcal I(r))\}$, the limit
$$
\lim_{m\rightarrow\infty}\frac{e_s(\mfa,R/I(1)_{mn_1}\cdots I(r)_{mn_r})}{m^{d-s}/(d-s)!}
$$
 is a homogeneous real polynomial $H_s(n_1,\ldots,n_r)$ of degree $d-s$  for $n_1,\ldots,n_r\in \NN$, which allows us to define  the mixed multiplicities 
 $$
 e_s(\mathcal I(1)^{[d_1]},\ldots,\mathcal I(r)^{[d_r]})
 $$
   from the coefficients of this polynomial in Definition \ref{Def4}. We have an associativity formula (\ref{M9}),
   $$
   e_s(\mfa,\mathcal I(1)^{[d_1]},\ldots,\mathcal I(r)^{[d_r]})
   =\sum_{\mfp}e_{R_{\mfp}}(\mathcal I(1)_{\mfp}^{[d_1]},\ldots,\mathcal I(r)_{\mfp}^{[d_r]})e_{\mfa}(R/\mfp)
   $$
   where the sum is over all $\mfp\in\mbox{Spec}(R)$ such that $\dim R/\mfp=s$ and $\dim R_{\mfp}=d-s$.

    We define in Proposition \ref{Prop2} and Definition \ref{Def4} the multiplicities $e_s(\mfa,\mathcal I;N)$  and mixed multiplicities $e_s(\mfa,\mathcal I(1),\ldots,\mathcal I(r);N)$  for arbitrary finitely generated $R$-modules $N$.

 In Section \ref{SecDiv}, we define divisorial filtrations and  $s$-divisorial filtrations. We prove the converse to the Rees Theorem \ref{Rees1*} for $s$-divisorial filtrations in Theorem \ref{Theorem9}. We prove the characterization of equality in the Minkowski inequality (\ref{MIEQ}) for $s$-divisorial filtrations in Theorem \ref{Theorem10}.
 
  In Section \ref{SecBound}, we define bounded filtrations and bounded $s$-filtrations. We prove the converse to the Rees Theorem \ref{Rees1*} for bounded $s$-filtrations in Theorem \ref{RBT}. We prove the characterization of equality in the Minkowski inequality (\ref{MIEQ}) for bounded $s$-filtrations in Theorem \ref{TRSKT}.
  
  In Theorems \ref{Theorem9}, \ref{Theorem10},  \ref{RBT} and \ref{TRSKT}, we have the assumption that $R$ is an excellent local domain. These theorems generalize the corresponding theorems for divisorial and bounded filtrations of $m_R$-primary ideal of \cite[Corollary 7.4]{C3}, \cite[Theorem 12.1]{C3}, \cite[Theorem 13.1]{C3} and Theorem \cite[Theorem 13.2]{C3}.
  
  In Examples \ref{Example1} and \ref{Example2} we show that Theorems \ref{RBT} and \ref{TRSKT} do not extend to arbitrary bounded or divisorial filtrations (when $s>0$ so that the filtrations do not consist of $m_R$-primary ideals).
  
 Let $R$ be a local ring and $\mathcal I=\{I_n\}$ be a filtration of  $R$. Define the graded $R$-algebra
  $$
  R[\mathcal I]=\sum_{m\ge 0}I_mt^m
  $$
  and let $\overline {R[\mathcal I]}$ be the integral closure of $R[\mathcal I]$ in the polynomial ring $R[t]$.

 In Section \ref{SecDiv}, we define divisorial filtrations. Suppose that $R$ is  a  local domain. Let $\nu$ be a divisorial valuation of the quotient field of $R$ which is nonnegative on $R$. We have the valuation ideals
 $$
I(\nu)_m=\{f\in R\mid \nu(f)\ge m\}
$$
for $m\in \NN$. The prime ideal $\mfp=I(\nu)_1$ is called the center of $\nu$ on $R$. We say that $\nu$ is an $s$-valuation if $\dim R/\mfp=s$.

A divisorial filtration of $R$ is a filtration $\mathcal I=\{I_m\}$ such that  there exist divisorial valuations $\nu_1,\ldots,\nu_r$ and $a_1,\ldots,a_r\in \RR_{\ge0}$ such that for all $m\in \NN$,
$$
I_m=I(\nu_1)_{\lceil ma_1\rceil }\cap\cdots\cap I(\nu_r)_{\lceil ma_r\rceil}.
$$
An $s$-divisorial filtration of $R$ is a filtration $\mathcal I=\{I_m\}$ such that  there exist $s$-valuations $\nu_1,\ldots,\nu_r$ and $a_1,\ldots,a_r\in \RR_{\ge0}$ such that for all $m\in \NN$,
$$
I_m=I(\nu_1)_{\lceil ma_1\rceil }\cap\cdots\cap I(\nu_r)_{\lceil ma_r\rceil}.
$$

 Observe that the trivial filtration $\mathcal I=\{I_m\}$, defined by $I_m=R$ for all $m$, is a degenerate case of a divisorial filtration and is a degenerate case of an $s$-divisorial filtration for all $s$. 
The  nontrivial $0$-divisorial filtrations are the  divisorial $m_R$-filtrations of \cite{C3}.

 We will often denote a divisorial filtration $\mathcal I$ on a local domain $R$ by $\mathcal I=\mathcal I(D)$. The motivation for this notation comes from the concept of a representation of a divisorial filtration on an excellent local domain, defined before Theorem \ref{Theorem9}.

In Section \ref{SecBound}, we define bounded filtrations and bounded $s$-filtrations. 
If $\mathcal I(D)$ is  a divisorial filtration, then $R[\mathcal I(D)]$ is integrally closed in $R[t]$; that is, $\overline{R[\mathcal I(D)]}=R[\mathcal I(D)]$.

A filtration $\mathcal I=\{I_n\}$ on $R$ is said to be a bounded filtration if there exists a divisorial filtration $\mathcal I(D)$ on $R$ such that
$\overline{R[\mathcal I]}=R[\mathcal I(D)]$.
A filtration $\mathcal I=\{I_n\}$ on $R$ is said to be a bounded $s$-filtration if there exists an $s$-divisorial filtration $\mathcal I(D)$ on $R$ such that
$\overline{R[\mathcal I]}=R[\mathcal I(D)]$.

If $I$ is an ideal, then the  $I$-adic filtration is bounded (Lemma \ref{BdLemma3}). Further, the filtration of integral closures of powers of an ideal $I$, and symbolic filtrations are bounded.

    In Section \ref{SecIneq}, we extend some classical inequalities for multiplicities and mixed multiplicities of $m_R$-primary ideals to filtrations, generalizing the inequalities of \cite{CSS} for filtrations of $m_R$-primary ideals to arbitrary filtrations. We also extend some other classical inequalities for multiplicities of ideals to filtrations.   We first prove the following theorem, which is proven in Theorem \ref{Rees1}. This theorem  is proven for $m_R$-primary filtrations in \cite[Theorem 6.9]{CSS}, \cite[Theorem 1.4]{C4}.

\begin{Theorem}\label{Rees1*}(Theorem \ref{Rees1})
	Let $R$   be an analytically unramified local ring of dimension $d$, $N$ be a finitely generated $R$-module, $\mathfrak a$ be an $m_R$-primary ideal  and  $\mathcal I=\{I_n\},$ $\mathcal J=\{J_n\}$ be filtrations of ideals of $R$ with $\mathcal J\subset \mathcal I.$ Suppose $R[\mathcal I]$ is integral over $R[\mathcal J]$. Then 
	\begin{itemize}
		\item[$(i)$] $s(\mathcal I)=s(\mathcal J),$
		\item[$(ii)$] $e_s(\mathfrak a,\mathcal I;N)=e_s(\mathfrak a,\mathcal J;N)$ for all  $s$ such that $s(\mathcal I)=s(\mathcal J)\leq s\leq d.$ 
	\end{itemize}
\end{Theorem}

The converse of Theorem \ref{Rees1*} fails for arbitrary filtrations, as shown by a simple example of filtrations of $m_R$-primary ideals in \cite{CSS}. A famous theorem of Rees \cite{R} (also \cite[Theorem 11.3.1]{HS}) shows that if $R$ is a formally equidimensional local ring and $J\subset I$ are $m_R$-primary ideals, then the converse of Theorem \ref{Rees1*} does hold for the $J$-adic and $I$-adic filtrations $\mathcal J=\{J^n\}\subset \mathcal I=\{I^n\}$. In this situation, the Rees algebra of $I$, $\bigoplus_{n\geq 0}I^n$ is integral over the Rees algebra of $J$, $\bigoplus_{n\geq 0}J^n$, if and only if the integral closures of ideals  $\overline I=\overline J$ are equal, which is the condition of Rees's theorem. In Theorem \ref{RBT}, we prove that if $\mathcal J$ is a bounded $s$-filtration and $\mathcal I$ is an arbitrary filtration with $\mathcal J\subset \mathcal I$, then the converse of Theorem \ref{Rees1*} holds. 
This generalizes the theorem for bounded filtrations of $m_R$-primary ideals in \cite[Theorem 13.1]{C3}.

The Minkowski inequalities were formulated and proven for   $m_R$-primary ideals in reduced equicharacteristic zero local rings by Teissier \cite{T1}, \cite{T2} and proven in full generality for  $m_R$-primary ideals in   arbitrary local rings,  by Rees and Sharp \cite{RS}. The same inequalities hold for filtrations. They were proven for $m_R$-filtrations in local rings $R$ such that $\dim N(\hat R)<\dim R$ in \cite[Theorem 6.3]{CSS}. We prove them for arbitrary filtrations in an analytically unramified local ring in Theorem \ref{Mink1}. In the following theorem, we state the fundamental inequality from which all other inequalities follow, (i) of Theorem \ref{Mink1*}, and the most famous inequality (ii) (The Minkowski Inequality).

\begin{Theorem}\label{Mink1*} (Minkowski Inequalities)  Let $R$  be an analytically unramified  local ring of dimension $d,$   $N$ be a finitely generated $R$-module, $\mathfrak a$ be an $m_R$-primary ideal,      $\mathcal I=\{I_j\}$ and $\mathcal J=\{J_j\}$ be filtrations of ideals of $R.$ Let $\max\{s(\mathcal I), s(\mathcal J)\}\leq s< d$ and $k:=d-s$.
	\begin{itemize}
		\item[$(i)$] Let $k\geq 2.$ For $1\le i\le k-1,$ $$e_s(\mathfrak a,\mathcal I^{[i]},\mathcal J^{[k-i]};N)^2\le e_s(\mathfrak a,\mathcal I^{[i+1]},\mathcal J^{[k-i-1]};N)e_s(\mathfrak a, \mathcal I^{[i-1]},\mathcal J^{[k-i+1]};N),$$
		\item[$(ii)$]  $s(\mathcal I\mathcal J)=\max\{s(\mathcal I), s(\mathcal J)\}$ and $e_{s}(\mathfrak a,\mathcal I\mathcal J;N)^{\frac{1}{k}}\le e_{s}(\mathfrak a,\mathcal I;N)^{\frac{1}{k}}+e_{s}(\mathfrak a,\mathcal J;N)^{\frac{1}{k}}$, where $\mathcal I\mathcal J=\{I_jJ_j\}$.
	\end{itemize}
\end{Theorem}
   
There is a  characterization of when equality holds for $m_R$-primary ideals in the Minkowski Inequality by 
Teissier \cite{T3} (for Cohen-Macaulay normal two-dimensional complex analytic $R$), Rees and Sharp \cite{RS} (in dimension 2) and Katz \cite{Ka} (in complete generality).

They have shown that 
 if $R$ is a formally equidimensional  local ring and $I, J$  are $m_R$-primary ideals then the Minkowski inequality
 $$
 e_R(IJ)^{\frac{1}{d}}=e_R(I)^{\frac{1}{d}}+e_R(J)^{\frac{1}{d}}
 $$
 holds if and only if there exists $a,b\in \ZZ_{>0}$ such that the integral closures $\overline {I^a}=\overline{J^b}$ are equal.
 This condition is equivalent to the statement that the integral closures of the Rees algebras of $I^a$ and $J^b$ are equal;
 %
that is,  
there exist positive integers $a$ and $b$ such that 
\begin{equation}\label{Mink2**}
\overline{\sum_{n\ge 0}I^{an}t^{n}}=\overline{\sum_{n\ge 0}J^{bn}t^{n}}.
\end{equation}



%

  We show in Theorem \ref{TheoremN10} that if   $\mathcal I$ and $\mathcal J$ are filtrations on an analytically unramified local ring $R$ and there exist $a,b\in \ZZ_{>0}$ such that 
 the integral closures 
 of the $R$-algebras $\oplus_{n\ge 0}I_{an}t^n$ and $\oplus_{n\ge 0}J_{bn}t^n$ are  
 equal, then the Minkowski equality 
 \begin{equation}\label{MIEQ}
e_{s}(\mathfrak a,\mathcal I\mathcal J;N)^{\frac{1}{d-s}}= e_{s}(\mathfrak a,\mathcal I;N)^{\frac{1}{d-s}}+e_{s}(\mathfrak a,\mathcal J;N)^{\frac{1}{d-s}}
\end{equation}
holds. However,  
if  $\mathcal I$ and $\mathcal J$ are filtrations on an analytically unramified local ring $R$ such that  the Minkowski Equality 
(\ref{MIEQ}) holds, then in general, the integral closures of the $R$-algebras $\oplus_{n\ge 0}I_{an}t^n$ and $\oplus_{n\ge 0}J_{bn}t^n$ are not equal for all $a,b\in \ZZ_{>0}$, even for filtrations of $m_R$-primary ideals in a regular local ring (so that $s=0$), as is shown in a simple example in \cite{CSS}. 

 In Theorem \ref{TRSKT}, we show that if $\mathcal I(1)$ and $\mathcal I(2)$ are two nontrivial bounded $s$-filtrations in an excellent local domain $R$, then  the Minkowski Equality holds if and only if   there exist positive integers $a$ and $b$ such that there is equality of integral closures 
$$
\overline{\sum_{n\ge 0}I(1)_{an}t^n}= \overline{\sum_{n\ge 0}I(2)_{bn}t^n},
  $$
   giving a complete generalization of the Teissier, Rees and Sharp, Katz Theorem for bounded $s$-filtrations. This theorem was proven for bounded filtrations of $m_R$-primary ideals in \cite[Theorem 13.2]{C3}.

 In Lemma \ref{FiltIneq}, we prove that if $\mathcal J(i)\subset \mathcal I(i)$ are filtrations for $1\le i\le r$ on an analytically unramified local ring and $N$ is a finitely generated $R$-module, then we have inequalities of mixed multiplicities
  $$
 e_s(\mfa,\mathcal I(1)^{[d_1]},\ldots,\mathcal I(r)^{[d_r]};N)
 \le  e_s(\mfa,\mathcal J(1)^{[d_1]},\ldots,\mathcal J(r)^{[d_r]};N)
 $$
 for $s\ge\max\{s(\mathcal J(1)),\ldots,s(\mathcal J(r))\}$.
 
 In Proposition \ref{PropUS},  we generalize a formula on  multiplicity of specialization of ideals (\cite{HSV}, \cite[formula (2.1)]{Li}) to filtrations.

 In Theorem \ref{BoGen}, we extend a theorem of B\"oger (\cite{EB}, \cite[Corollary 11.3.2]{HS}) about equimultiple ideals in a formally equidimensional local ring to
a theorem about bounded $s$-filtrations in an excellent local domain. 

An ideal $I$ in a local ring $R$ is equimultiple if $\mbox{ht}(I)=\ell(I)$ where $\ell(I)$ is the analytic spread of $I$. In an excellent local domain, our theorem is strictly stronger that Bo\"ger's theorem, even for $I$-adic filtrations. We show in  Corollary \ref{CorAS31*} that the $I$-adic filtration of an equimultiple ideal is a bounded $s$-filtration where $s=\dim R-\mbox{ht}(I)$. 
In Example \ref{ExAS1}, we show that there are ideals $I$ whose $I$-adic filtration is a bounded $s$-filtration, but $I$ is not equimultiple.

   \section{Notation}

We will denote the nonnegative integers by $\NN$ and the positive integers by $\ZZ_{>0}$,   the set of nonnegative rational numbers  by $\QQ_{\ge 0}$  and the positive rational numbers by $\QQ_{>0}$. 
We will denote the set of nonnegative real numbers by $\RR_{\ge0}$ and the positive real numbers by $\RR_{>0}$. For a real number $x$, $\lceil x\rceil$ will denote the smallest integer that is $\ge x$ and $\lfloor x\rfloor$ will denote the largest integer that is $\le x$. If $E_1,\ldots, E_r$ are prime divisors on a normal scheme $X$ and $a_1,\ldots,a_r\in \RR$, then $\lfloor \sum a_iE_i\rfloor$ denotes the integral divisor $\sum \lfloor a_i\rfloor E_i$ and $\lceil \sum a_iE_i\rceil$ denotes the integral divisor $\sum \lceil a_i\rceil E_i$.

A local ring is assumed to be Noetherian.
The maximal ideal of a local ring $R$ will be denoted by $m_R$. The quotient field of a domain $R$ will be denoted by ${\rm QF}(R)$. We will denote the length of an $R$-module $M$ by $\ell_R(M)$. Excellent local rings have many excellent properties which are enumerated in \cite[Scholie IV.7.8.3]{EGA}. We will make use of some of these properties without further reference.

\section{Filtrations}

A filtration $\mathcal I=\{I_n\}_{n\in\NN}$ of ideals on a ring $R$ is a descending chain
$$
R=I_0\supset I_1\supset I_2\supset \cdots
$$
of ideals such that $I_iI_j\subset I_{i+j}$ for all $i,j\in \NN$. A filtration $\mathcal I=\{I_n\}$ of ideals on a local ring $(R,\mR)$ is a filtration of $R$ by $\mR$-primary ideals if $I_n$ is $m_R$-primary for $n\ge 1$.
A filtration $\mathcal I=\{I_n\}_{n\in\NN}$ of ideals on a ring $R$ is called a Noetherian filtration if $\bigoplus_{n\ge 0}I_n$ is a finitely generated $R$-algebra.

If $I\subset R$ is an ideal, then $V(I)=\{\mfp\in \mbox{Spec}(R)\mid I\subset \mfp\}$. If $\mathcal I=\{I_n\}$ and $\mathcal J=\{J_n\}$ are filtrations of $R$, then we will write $\mathcal I\subset\mathcal J$ if $I_n\subset J_n$ for all $n$.

For any filtration  $\mathcal I=\{I_n\}$ and $\mfp\in \Spec R,$  let $\mathcal I_\mfp$ denote the filtration $\mathcal I_\mfp=\{I_{n}R_\mfp\}.$

 Let $R$ be a local ring and $\mathcal I=\{I_n\}$ be a filtration of  $R$. As defined in the introduction,  the graded $R$-algebra
  $$
  R[\mathcal I]=\sum_{m\ge 0}I_mt^m
  $$
  and  $\overline {R[\mathcal I]}$ is the integral closure of $R[\mathcal I]$ in the polynomial ring $R[t]$.

	\begin{Lemma}\label{dim}
		Let $R$ be a local ring and $\mathcal I=\{I_n\}$ be a filtration of ideals of $R.$  The following hold.
		\begin{itemize}
			\item[$(i)$] For all $n\geq 1,$ $V(I_1)=V(I_n)$ and $\dim R/I_1=\dim R/I_n.$ 
			\item[$(ii)$] If $\mathcal J=\{J_n\}$ is a filtration of ideals of $R$ such that $\mathcal J\subset \mathcal I$ and the $R$-algebra $R{[\mathcal I]}$ is a finitely generated $R{[\mathcal J]}$-module then $V(I_n)=V(J_m)$ for all $n,m\geq 1.$
		\end{itemize}	
	\end{Lemma}

\begin{proof}
	$(i)$ Since $\{I_n\}$ is a filtration, for all integers $n>m\geq 1,$ we have $I_m^n\subset I_n\subset I_m.$ Therefore  $V(I_1)=V(I_n)$ for all  $n\geq 1.$ Hence $\min \Ass R/I_1=\min \Ass R/I_n$ for all  $n\geq 1.$ Thus  for all  $n\geq 1,$   $\dim R/I_1=\dim R/I_n.$ 
	\\$(ii)$ Let $\{\alpha_1,\ldots,\alpha_r\}$ be a generating set of the $R$-algebra $R{[\mathcal I]}$ as an $R{[\mathcal J]}$-module. Suppose $\deg\alpha_i=d_i$ for all $i=1,\ldots,r$ and $d=\max\{d_1,\ldots,d_r\}.$ Then for all $n\geq d+1,$ we have $J_n\subset I_n\subset J_{n-d}.$ Since by $(1),$ $V(J_m)=V(J_1)$ for all $m\geq 1,$ we have $V(I_n)=V(J_n).$ Again by using $(1),$ we get $V(I_n)=V(J_m)$ for all $n,m\geq 1.$
\end{proof}

\begin{Definition} Let  $R$ be a  local ring and $\mathcal I=\{I_n\}$ be a filtration of ideals of $R$.  We define the {\it{dimension of the filtration $\mathcal I$}} to be  $s(\mathcal I)=\dim R/I_n$ (for any $n\geq 1$).
\end{Definition}

 The dimension $s(\mathcal I)$ 	 is well-defined by Lemma \ref{dim} $(i)$. In the case of the trivial filtration $\mathcal I=\{I_n\}$, where $I_n=R$ for all $n$, we have that $s(\mathcal I)=-1$.

\begin{Definition}
 Let  $R$  be a $d$ dimensional  local ring and $\mathcal I=\{I_n\}$ be a filtration of ideals of $R$.
For $s\in \NN$, we define 
$$
A({\mathcal I})=\min \Ass R/I_1\bigcap \{\mfp\in\Spec R: \dim R/\mfp=s\},
$$
the set of minimal primes $\mathfrak p$ of $I_1$ such that $\dim R/\mfp=s$.

\end{Definition}

\begin{Example}\label{Example3}The embedded associated primes that appear in the ideals in a filtration $\mathcal I=\{I_n\}$ may be infinite in number.
\end{Example}

 A simple example is as follows. Let $k$ be an infinite field, and let   $\{\alpha_i\}_{i\in \ZZ_{>0}}$ be a countable set of distinct elements of $k$. Let $R=k[x,y,z]_{(x,y,z)}$ be the localization of a polynomial ring over $k$ in three variables. Let $I_n=z^{n+1}(z,x-\alpha_ny)$ for $n\in\ZZ_{>0}$. $\mathcal I=\{I_n\}$ is thus a filtration. The associated primes of 
$I_n=(z^{n+1})\cap(z^{n+2},x-\alpha_ny)$ are $(z)$ and $(z,x-\alpha_ny)$. This is in contrast to the fact that the associated primes of the filtration of powers of an ideal $I$ in a  local ring is a finite set \cite{Br}.

\begin{Lemma}\label{LemmaIC1} Let $R$ be a local ring  and suppose that $\mathcal I=\{I_n\}$ is a filtration of $R$. Then the following are equivalent.
\begin{enumerate}
\item[1)] $R[t]$ is integral over $\sum_{n\ge 0}I_nt^n$. 
\item[2)] $I_1=R$. 
\item[3)] $\mathcal I$ is the trivial filtration.
\end{enumerate}
\end{Lemma}

\begin{proof} Suppose that $t$ is integral over $\sum_{n\ge 0}I_nt^n$. Then there exist $n\in \ZZ_{>0}$ and $a_i\in I_i$ for $0\le i\le n$ such that $t^n+(a_1t)t^{n-1}+\cdots+(a_nt^n)=0$. Thus $1\in (a_1,\ldots,a_n)\subset I_1$.
\end{proof}

Suppose that $R$ is a local ring and $\mathcal I=\{I_n\}$ is a filtration of $R$. The integral closure $\overline{\sum_{n\ge 0}I_nt^n}$ of $\sum_{n\ge 0}I_nt^n$ in $R[t]$ is a graded $R$-algebra 
$\overline{\sum_{n\ge 0}I_nt^n}=\sum_{n\ge 0}K_nt^n$,  where $\mathcal K=\{K_n\}$ is a filtration of $R$ (by \cite[Theorem 2.3.2]{HS}).

If $I$ is an ideal in a local ring $R$, let $\overline I$ denote its integral closure. 

\begin{Lemma}\label{BdLemma1} Let $R$ be a local ring and $\mathcal I=\{I_n\}$ be a filtration. Then
$$
\overline{R[\mathcal I]}=\sum_{m\ge 0}J_mt^m
$$
where $\{J_m\}$ is the filtration
$$
J_m=\{f\in R\mid f^r\in \overline{I_{rm}}\mbox{ for some }r>0\}.
$$
\end{Lemma}

The proof of  Lemma \ref{BdLemma1} for $m_R$-filtrations in \cite[Lemma 5.5]{C3} extends immediately to arbitrary divisorial filtrations. 

\begin{Remark}\label{Rem1} If $\mathcal I=\{I^n\}$ is the adic-filtration of the powers of a fixed ideal $I$, then $J_n=\overline{I^n}$ for all $n$.
\end{Remark}

\begin{Lemma}\label{LemmaIC2} Suppose that $R$ is a local ring, $\mathcal I=\{I_n\}$ is a filtration of $R$ and $p\in \mbox{Spec}(R)$. Let $\overline{R[\mathcal I]}=\oplus_{n\ge 0}K_n$. Then the integral closure $\overline{\sum_{n\ge 0}I_nR_pt^n}$ of $\sum_{n\ge 0}I_nR_pt^n$ in $R_p[t]$ is $\sum_{n\ge 0}K_nR_pt^n$.
\end{Lemma}

\begin{Lemma}\label{LemmaIC3} Let $R$ be a local ring,  $\mathcal I=\{I_n\}$ be a filtration of $R$ and $\overline{R[\mathcal I]}=\oplus_{n\ge 0}K_n$. Let $\mathcal K=\{K_n\}$. Then
\begin{enumerate}
\item[1)] $V(I_1)=V(K_1)$.
\item[2)] $s(\mathcal I)=s(\mathcal  K)$.
\end{enumerate}
\end{Lemma}

\begin{proof} By Lemmas \ref{LemmaIC1} and \ref{LemmaIC2}, $p\not\in V(I_1)$ if and only if $\overline{\sum_{n\ge 0}I_nR_pt^n}
=R_p[t]$ which holds if and only if $\sum_{n\ge 0}K_nR_pt^n=R_p[t]$, and this last condition holds if and only if $p\not\in V(K_1)$.
\end{proof}

  \section{Multiplicities of filtrations}    
    
 	Let $\mfa$ be an $\mR$-primary ideal of $R$ and $N$ be a finitely generated $R$-module with $\dim N=r.$ Define $$e_{\mfa}(N)=\lim_{k\to\infty}\frac{l_R(N/\mfa ^kN)}{k^r/r!}.$$ If $s\geq r=\dim N,$ define (\cite[V.2]{Se}, \cite[4.7]{BH})
\[ e_s({\mfa},N)= \left\{
\begin{array}{l l}
e_{\mfa}(N) & \quad \text{if $\dim N=s$ }\\
0 & \quad \text{if $\dim N<s.$ }
\end{array} \right.\] 
 
\begin{Example} The function $e_{\mfa}(N)$ of an $m_R$-primary ideal $\mfa$ does not extend to a function   $e_{\mathcal A}(N)$ of a filtration $\mathcal A=\{a_n\}$ of $m_R$-primary ideals on finitely generated $R$-modules $N$, even on a regular local ring. 
\end{Example}

The existence of such an example follows from \cite[Example 5.3]{C1}. Let $k$ be a field and $R$ be the $d$ dimensional power series ring over $k$, $R=k[[x_1,\ldots x_{d-1},y]]$. In \cite[Example 5.3]{C1}, a
 filtration $\mathcal A=\{a_n\}$ of $m_R$-primary ideals is constructed such that if $\mfp$ is the prime ideal $\mfp=(y)$ of $R$, then the limit
 $$
 \lim_{k\rightarrow\infty}\frac{ \ell_R((R/\mfp^m)/a_k(R/\mfp^m))}{k^{d-1}/(d-1)!}
 $$
 does not exist for any $m\ge 2$. In the example, a function $\sigma:\ZZ_+\rightarrow \QQ_+$ is constructed such that letting
 $N_n=(x_1,\ldots,x_{d-1})^n$, and defining $a_n=(N_n,yN_{n-\sigma(n)},y^2)$, we have that $\mathcal A=\{a_n\}$ is a filtration of $m_R$-primary ideals on $R$ for which the above limits do not exist.

Let $R$ be an analytically unramified local ring of dimension $d$, $N$ be a finitely generated $R$-module  and $\mathcal I=\{I_n\}$ is a filtration of $m_R$-primary ideals.  Then the multiplicity of $\mathcal I$ is defined by
$$
e_R(\mathcal I,N):=\lim_{m\rightarrow\infty}\frac{\ell_R(N/I_mN)}{m^d/d!}.
$$
This limit exists by \cite[Theorem 1.1]{C1} and \cite[Theorem 6.6]{CSS}. We further define the multiplicity of the trivial filtration $\mathcal I=\{I_n\}$, where $I_n=R$ for all $n$, to be $e_R(\mathcal I,N)=0$. We write $e_R(\mathcal I)=e_R(\mathcal I,R)$.

Suppose that $\mathcal I$ is a filtration of $R$ and that $s\ge s(\mathcal I)$. Let $N$ be a finitely generated $R$-module. 
Suppose that $\mathfrak p$ is a prime ideal of $R$ such that $\dim R/\mathfrak p=s$. Then $\dim R_{\mathfrak p}\le d-s$ with equality if $R$ is equidimensional and universally catenary. The universal catenary condition holds, for instance, if $R$ is regular or $R$ is excellent. Suppose that $\mathfrak p\in\mbox{Spec}(R)$ satisfies $\dim R/\mathfrak p=s$. Define the filtration $\mathcal I_{\mathfrak p}=\{I_nR_{\mathfrak p}\}$. Then
we have that
\begin{equation}\label{M11}
e_{R_p}(\mathcal I_{\mathfrak p},N_{\mfp})=\lim_{n\rightarrow\infty}\frac{\ell_{N_{\mathfrak p}}(R_{\mathfrak p}/I_nN_{\mathfrak p})}{n^{{\rm dim}\, R_{\mathfrak p}}/\dim R_{\mathfrak p}!}
\end{equation}
exists.

We will frequently make use of the fact  that if $R$ is a  local ring which is analytically unramified and $\mathfrak p\in \mbox{Spec}(R)$ is a prime ideal, then $R_{\mathfrak p}$ is analytically unramified (\cite[Proposition 9.1.4]{HS}).

\begin{Proposition}\label{Prop2} Suppose that $R$ is an analytically unramified local ring, $N$ is a finitely generated $R$-module and $\mathfrak a$ is an $m_R$-primary ideal  and $\mathcal I$ is a filtration on $R$. Suppose that $s\in \NN$ is such that $s(\mathcal I)\le s\le d$. Then 
the limit
$$
e_s(\mathfrak a,\mathcal I;N):=\lim_{m\rightarrow\infty}\frac{e_s(\mathfrak a,N/I_mN)}{m^{d-s}/(d-s)!}
$$
exists. Further, 
\begin{equation}\label{M10*}
e_s(\mathfrak a,\mathcal I;N)=\sum_{\mathfrak p}e_{R_{\mathfrak p}}(\mathcal I_{\mathfrak p},N_{\mfp})e_{\mathfrak a}(R/\mathfrak p)
\end{equation}
where the sum is over all $\mathfrak p\in {\Spec}(R)$ such that $\dim R/\mathfrak p=s$ and $\dim R_{\mathfrak p}=d-s$.
\end{Proposition}
\begin{proof}
	Let $\mfp$ be a prime ideal of $R$ such that $\dim R/\mfp=s.$ If $\mfp\notin A(\mathcal I)$  
	 then  $\ell_{R_{\mfp}}(N_{\mfp}/I_nN_{\mfp})=0.$ If $\mfp\in A({\mathcal I})$ 
	  then for all $n\geq 1,$ $I_nR_{\mfp}$ are $\mfp R_{\mfp}$-primary ideals of $R_{\mfp}.$ Hence for all  $\mfp\in A({\mathcal I}),$ 
	   the limit  $$\lim\limits_{m\to\infty}\frac{\ell_{R_{\mfp}}(N_\mfp/I_mN_\mfp)}{m^{\dim R_\mfp}/\dim R_\mfp !}$$ exists by (\ref{M11}) and since $R_{\mathfrak p}$ is analytically unramified.
	Since $A({\mathcal I})$ 
	 is a finite set, using \cite[Corollary 4.7.8]{BH}, we get that the limit
	\begin{eqnarray*}
		\lim\limits_{m\to\infty}\frac{e_s(\mathfrak a,N/I_mN)}{m^{d-s}/(d-s)!}&=& \lim\limits_{m\to\infty}\frac{(d-s)!}{m^{d-s}}\sum\limits_{\substack{\mfp\in\Spec R,\\ \dim R/\mfp=s}}{\ell_{R_{\mfp}}(N_\mfp/I_mN_\mfp)}e_{\mathfrak a}(R/\mfp)\\&=& \sum\limits_{\substack{\mfp\in\Spec R,\\ \dim R/\mfp=s}}\Big(\lim\limits_{m\to\infty}\frac{\ell_{R_{\mfp}}(N_\mfp/I_mN_\mfp)}{m^{d-s}/(d-s)!}\Big)e_{\mathfrak a}(R/\mfp)\\&=& \sum\limits_{\substack{\mfp\in\Spec R,\\ \dim R/\mfp=s\mbox{ and } \dim R_{\mathfrak p}=d-s}} e_{R_{\mathfrak p}}(\mathcal I_{\mathfrak p},N_{\mfp})e_{\mathfrak a}(R/\mathfrak p)
	\end{eqnarray*} exists.
\end{proof}

If $\mathcal I$ is a filtration of $m_R$-primary ideals, then $s(\mathcal I)=0$ and
$$
e_0(\mathfrak a, \mathcal I;N)=e_R(\mathcal I;N)e_{\mathfrak a}(R/m_R)=e_R(\mathcal I;N).
$$
 We will write $e_s(\mfa,\mathcal I)=e_s(\mfa,\mathcal I;R)$.	

\subsection{Mixed multiplicities of filtrations} Let $R$ be an analytically unramified local ring of dimension $d$ and $N$ be a finitely generated $R$-module. 
Suppose that
 $\mathcal I(1)=\{I(1)_i\},\ldots,\mathcal I(r)=\{I(r)_i\}$ are filtrations of $m_R$-primary ideals. It is shown in \cite[Theorem  6.6]{CSS} that the function
\begin{equation}\label{M2}
P(n_1,\ldots,n_r)=\lim_{m\rightarrow \infty}\frac{\ell_R(N/I(1)_{mn_1}\cdots I(r)_{mn_r}N)}{m^d/d!}
\end{equation}
is  a homogeneous polynomial  of total degree $d$ with real coefficients for all  $n_1,\ldots,n_r\in\NN$.   The mixed multiplicities $e_R(\mathcal I(1)^{[d_1]},\ldots,\mathcal I(r)^{[d_r]};N)$ of $N$ of type $(d_1,\ldots,d_r)$ with respect to the filtrations $\mathcal I(1),\ldots,\mathcal I(r)$ are defined 
 from the coefficients of $P$, generalizing the definition of mixed multiplicities for $m_R$-primary ideals. Specifically,   
 we write 
\begin{equation}\label{M3}
P(n_1,\ldots,n_r)=\sum_{d_1+\cdots +d_r=d}\frac{d!}{d_1!\cdots d_r!}e_R(\mathcal I(1)^{[d_1]},\ldots, \mathcal I(r)^{[d_r]};N)n_1^{d_1}\cdots n_r^{d_r}.
\end{equation}
We have need of  the formulas (\ref{M2}) and (\ref{M3}) in  the slightly more general case that   $\mathcal I(1)=\{I(1)_i\},\ldots,\mathcal I(r)=\{I(r)_i\}$ are filtrations of $R$
such that for each $j$, $\mathcal I(j)$ is either a filtration of $m_R$-primary ideals or $\mathcal I(j)$ is the trivial filtration $I(j)_n=R$ for all $n\in R$. In this case, there exists $\sigma\in \NN$ with $0\le \sigma\le r$ such that either $\sigma=0$ 
and $\mathcal I(j)$ are filtrations of $m_R$-primary ideals for all $j$
or $\sigma>0$ and there exist $1\le i_1<\cdots <i_{\sigma}\le r$ such that $\mathcal I(j)$ is a filtration of $m_R$-primary ideals if $j\in \{i_1,\ldots,i_{\sigma}\}$ and $\mathcal I(j)$ is the trivial filtration $I(j)_n=R$ for all $n\in \NN$ if $j\not\in \{i_1,\ldots,i_{\sigma}\}$. As in (\ref{M2}), we define
\begin{equation}\label{M4}
P(n_1,\ldots,n_r)=\lim_{m\rightarrow \infty}\frac{\ell_R(N/I(1)_{mn_1}\cdots I(r)_{mn_r}N)}{m^d/d!}
\end{equation}
In this case, we have that 
$$
P(n_1,\ldots,n_r)=P(0,\ldots,0,n_{i_1},0,\ldots,0,n_{i_2},0,\ldots,0,n_{i_{\sigma}},0,\ldots,0)
$$
which is a homogeneous polynomial of degree $d$ by (\ref{M2}) and (\ref{M3}). In this case, we also define the mixed multiplicities $e_R(\mathcal I(1)^{[d_1]},\ldots,\mathcal I(r)^{[d_r]};N)$ of $N$ of type $(d_1,\ldots,d_r)$ with respect to the filtrations $\mathcal I(1),\ldots,\mathcal I(r)$ 
 from the coefficients of $P$, so that $P$ 
has an expansion
\begin{equation}\label{M5}
P(n_1,\ldots,n_r)=\sum_{d_1+\cdots +d_r=d}\frac{d!}{d_1!\cdots d_r!}e_R(\mathcal I(1)^{[d_1]},\ldots, \mathcal I(r)^{[d_r]};N)n_1^{d_1}\cdots n_r^{d_r}.
\end{equation}
We have that $P(n_1,\ldots,n_r)$ is a homogeneous polynomial of degree $d$ in the variables $n_{i_1},\ldots,n_{i_\sigma}$. 
Thus 
\begin{equation}\label{M6}
e_R(\mathcal I(1)^{[d_1]},\ldots, \mathcal I(r)^{[d_r]};N)=0 \mbox{ if some $d_j>0$ with $j\not\in \{i_1,\ldots,i_{\sigma}\}$.}
\end{equation}

We will write $e_R(\mathcal I(1)^{[d_1]},\ldots, \mathcal I(r)^{[d_r]})=e_R(\mathcal I(1)^{[d_1]},\ldots, \mathcal I(r)^{[d_r]};R)$.

 \begin{Theorem}\label{Theorem3} Suppose that $R$ is an analytically unramifed local ring of dimension $d$, $N$ is a finitely generated $R$-module,
 $\mathfrak a$ is an $m_R$-primary ideal   and  that
 $\mathcal I(1)=\{I(1)_i\},\ldots,\mathcal I(r)=\{I(r)_i\}$ are filtrations of  ideals.  Suppose that $s\in \NN$ is such that  $\max\{s(\mathcal I(1)),\ldots,s(\mathcal I(r)\}\le s\le d$ and $n_1,\ldots,n_r\in \NN$.
  Then for $n_1,\ldots,n_r\in \NN$,
   \begin{equation}\label{M7}
 H_s(n_1,\ldots,n_r)=\lim_{m\rightarrow\infty}\frac{e_s(\mathfrak a,N/I(1)_{mn_1}\cdots I(r)_{mn_r}N)}{m^{d-s}/(d-s)!}
 \end{equation}
is a homogeneous polynomial of total degree $d-s$. 
\end{Theorem}
\begin{proof} Define
	$A=\bigcup\limits_{j=1}^rA({\mathcal I(j)}).$ 
	Then $A$ is a finite set.

For all $(n_1,\ldots,n_r)\in \mathbb N^r,$  consider the filtrations $\mathcal J(n_1,\ldots,n_r)=\{J(n_1,\ldots,n_r)_m=I(1)_{mn_1}\cdots I(r)_{mn_r}\}$ of ideals of $R.$  Note that 
	$V(J(n_1,\ldots,n_r)_m)=V(J(n_1,\ldots,n_r)_{m+1})$ for all $n_1,\ldots,n_r\in\mathbb N$ and $m\geq 1.$ Let $\mfp$ be a prime ideal of $R$ such that $\dim R/\mfp=s.$ If $\mfp\notin A$ then for all $m\geq 1,$ $\ell_{R_{\mfp}}(R_{\mfp}/J(n_1,\ldots,n_r)_mR_{\mfp})=0.$ 
	If $\mfp\in A$ then for all $m\geq 1,$ $J(n_1,\ldots,n_r)_mR_{\mfp}=I(1)_{mn_1}\cdots I(r)_{mn_r}R_{\mfp}$ are either $R_\mfp$ or $\mfp R_{\mfp}$-primary ideals of $R_{\mfp}.$ 
	Therefore by Proposition \ref{Prop2}, the limit
\begin{equation}\label{M10}
\begin{array}{lll}
\lim\limits_{m\to\infty}\frac{e_s(\mathfrak a,N/I(1)_{mn_1}\cdots I(r)_{mn_r}N)}{m^{d-s}/(d-s)!}&=& 
e_s(\mathcal J(n_1,\ldots,n_r);N)\\
&=&\sum_{\mathfrak p}e_{R_{\mathfrak p}}(\mathcal J(n_1,\ldots,n_r)_{\mathfrak p},N_{\mfp})e_{\mathfrak a}(R/\mathfrak p).
\end{array}
\end{equation}
where the sum is over $\mathfrak p\in \mbox{Spec}(R)$ such that $\dim R/\mathfrak p=s$ and $\dim R_{\mathfrak p}=d-s$.
 By (\ref{M4}), (\ref{M5}) and (\ref{M6}), we have for each $\mathfrak p$ that 
\begin{equation}\label{M12}
e_{R_{\mathfrak p}}(\mathcal J(n_1,\ldots,n_r)_{\mathfrak p},N_{\mfp})
=\sum_{d_1+\cdots+d_r=d-s}\frac{(d-s)!}{d_1!\cdots d_r!}e_{R_{\mathfrak p}}(\mathcal I(1)_{\mathfrak p}^{[d_1]},\ldots,\mathcal I(r)_{\mathfrak p}^{[d_r]};N_{\mfp})n_1^{d_1}\cdots n_r^{d_r}
\end{equation}
is a (possibly zero) homogeneous polynomial of degree $d-s$ in $n_1,\ldots,n_r$.
Thus the function (\ref{M7}) is a homogeneous polynomial in $n_1,\ldots,n_r$ of  total degree  $d-s$.
\end{proof}

\begin{Definition}\label{Def4}
With the assumptions of Theorem \ref{Theorem3}, the mixed multiplicities 
$$
e_s(\mathfrak a,\mathcal I(1)^{[d_1]},\ldots,I(r)^{[d_r]};N)
$$
 are defined from the expansion
\begin{equation}\label{M8}
H_s(n_1,\ldots,n_r)=
\sum_{d_1+\cdots +d_r=d-s}\frac{(d-s)!}{d_1!\cdots d_r!}e_s(\mathfrak a,\mathcal I(1)^{[d_1]},\ldots, \mathcal I(r)^{[d_r]};N)n_1^{d_1}\cdots n_r^{d_r}.
\end{equation}
 \end{Definition}

 \begin{Theorem}\label{Theorem5}
Let assumptions be as in Theorem \ref{Theorem3}.
Then we have the formula
\begin{equation}\label{M9}
e_s(\mathfrak a,\mathcal I(1)^{[d_1]},\ldots, \mathcal I(r)^{[d_r]};N)
=\sum_{\mathfrak p}e_{R_{\mathfrak p}}((\mathcal I(1)_{\mathfrak p})^{[d_1]},\ldots,(\mathcal I(r)_{\mathfrak p})^{[d_r]};N_{\mfp})e_{\mathfrak a}(R/\mathfrak p)
\end{equation}
where the sum is over all $\mathfrak p\in {\Spec}(R)$ such that $\dim R/\mathfrak p=s$ and $\dim R_{\mathfrak p}=d-s$.
\end{Theorem}
\begin{proof}
The theorem follows from equations (\ref{M10}) and (\ref{M12}).
	\end{proof}
	
	If $\mathcal I(1),\ldots,\mathcal I(r)$ are filtrations of $m_R$-primary ideals (or are trivial filtrations) then
	$$
	e_0(\mathfrak a,\mathcal I(1)^{[d_1]},\ldots,\mathcal I(r)^{[d_r]};N)=e_R(\mathcal I(1)^{[d_1]},\ldots,\mathcal I(r)^{[d_r]};N)e_{\mathfrak a}(R/m_R).
	$$
We will write $e_s(\mathfrak a,\mathcal I(1)^{[d_1]},\ldots,\mathcal I(r)^{[d_r]})=e_s(\mathfrak a,\mathcal I(1)^{[d_1]},\ldots,\mathcal I(r)^{[d_r]};R)$.
	
	We have that
$$e_s(\mathfrak a,\mathcal I(1)^{[0]},\ldots,\mathcal I(i-1)^{[0]},\mathcal I(i)^{[d-s]},\mathcal I(i+1)^{[0]},\ldots\mathcal I(r)^{[0]}, N)=e_s(\mathfrak a,\mathcal I(i),N)$$ for all $1\leq i\leq r.$

\section{Inequalities}\label{SecIneq}

Suppose that $R$ is a local ring and $I\subset J$ are $m_R$-primary ideals. Rees showed in \cite{R} that if $R[J]$ is integral over $R[I]$ then the multiplicities $e_I(R)=e_J(R)$ are equal and he proved the converse, that $e_I(R)=e_J(R)$ implies $R[J]$ is integral over $R[I]$ if $R$ is formally equidimensional. We show in \cite[Theorem 6.9]{CSS} and \cite[Appendix]{C3} that the first statement extends to arbitrary filtrations of $m_R$-primary ideals, in a local ring such that $\dim N(\hat R)<\dim R$. However, the converse does not hold for $m_R$-primary ideals, even if $R$ is a regular local ring. A simple example is given in \cite{CSS}. 
We extend this theorem to arbitrary filtrations in the following theorem. 

	\begin{Theorem}\label{Rees1}
		Let $R$ be an analytically unramified local ring of dimension $d$, $N$ be a finitely generated $R$-module, $\mathfrak a$ be an $m_R$-primary ideal  and  $\mathcal I=\{I_n\},$ $\mathcal J=\{J_n\}$ be filtrations of ideals of $R$ with $\mathcal J\subset \mathcal I.$ Suppose $R[\mathcal I]$ is integral over $R[\mathcal J]$. Then 
		\begin{itemize}
			\item[$(i)$] $s(\mathcal I)=s(\mathcal J),$
			\item[$(ii)$] $e_s(\mathfrak a,\mathcal I;N)=e_s(\mathfrak a,\mathcal J;N)$ for all  $s$ such that $s(\mathcal I)=s(\mathcal J)\leq s\leq d.$ 
		\end{itemize}
\end{Theorem}

\begin{proof} We have that $V(I_1)=V(J_1)$ and $s(\mathcal I)=s(\mathcal J)$ by Lemma \ref{LemmaIC3}. Thus by Proposition \ref{Prop2}, $e_s(\mathcal I;N)=e_s(\mathcal J;N)=0$ for $s>s(\mathcal I)=s(\mathcal J)$ and
we have that
$$
e_s(\mathfrak a,\mathcal I;N)=\sum_{\mathfrak p}e_{R_{\mathfrak p}}(\mathcal I_{\mathfrak p},N_{\mfp})e_{\mathfrak a}(R/\mathfrak p)\mbox{ and }
e_s(\mathfrak a,\mathcal J;N)=\sum_{\mathfrak p}e_{R_{\mathfrak p}}(\mathcal J_{\mathfrak p},N_{\mfp})e_{\mathfrak a}(R/\mathfrak p)
$$
where the sums are over prime ideals $\mathfrak p$ such that $\dim R/\mathfrak p=s$ and $\dim R_{\mathfrak p}=d-s$.
Suppose that $s=s(\mathcal I)$. 
By Lemma \ref{dim}, if  $\dim R/\mathfrak p=s$, then either $I_1R_{\mathfrak p}$ and $J_1R_{\mathfrak p}$ are both $m_{R_{\mathfrak p}}$-primary ideals or $I_1R_{\mathfrak p}=J_1R_{\mathfrak p}=R_{\mathfrak p}$. In the first case,
$\mathcal J_{\mathfrak p}\subset \mathcal I_{\mathfrak p}$ are filtrations of $m_{R_{\mathfrak p}}$-primary ideals and in the second case, $\mathcal I_{\mathfrak p}$ and $\mathcal J_{\mathfrak p}$ are both the trivial $R_{\mathfrak p}$-filtration.

Let $\mathfrak p_1,\ldots,\mathfrak p_r$ be the prime ideals in $R$ such that the first case holds. Then
$$
e_s(\mathfrak a,\mathcal I;N)=\sum_{i=1}^re_{R_{\mathfrak p_i}}(\mathcal I_{\mathfrak p_i},N_{\mfp_i})e_{\mathfrak a}(R/\mathfrak p_i)\mbox{ and }
e_s(\mathfrak a,\mathcal J;N)=\sum_{i=1}^re_{R_{\mathfrak p_i}}(\mathcal J_{\mathfrak p_i},N_{\mfp_i})e_{\mathfrak a}(R/\mathfrak p_i).
$$

We have that $R[\mathcal J_{\mathfrak p_i}]$ is integral over  $R[\mathcal I_{\mathfrak p_i}]$ for all $1\le i\le r$ and
$R_{\mathfrak p_i}$ is analytically unramified for all $i$ by \cite[Proposition 9.1.4]{HS}. Thus 
$e_{R_{\mathfrak p_i}}(\mathcal I_{\mathfrak p_i},N_{\mfp_i})=e_{R_{\mathfrak p_i}}(\mathcal J_{\mathfrak p_i},N_{\mfp_i})$ for $1\le i\le r$ by \cite[Theorem 6.9]{CSS} or \cite[Appendix]{C4}, and so $e_s(\mfa,\mathcal I;N)=e_s(\mfa,\mathcal J;N)$.
\end{proof}

 \begin{Corollary}\label{Ap1} Suppose that $R$ is an analytically unramified local ring, $\mathfrak a$ is an $m_R$-primary ideal  and $\mathcal I=\{I_n\}$ is a filtration of $R$. Let 
 $$
J_n=\{f\in R\mid f^r\in \overline{I_{rn}}\mbox{ for some }r>0\}.
$$ 
 Then
 $$
 s:=s(\mathcal I)=s(\{\overline{I_n}\})=s(\{J_n\})
 $$
 and
 $$
 e_s(\mathfrak a,\mathcal I)=e_s(\mathfrak a,\{\overline{I_n}\})=e_s(\mathfrak a,\{J_n\}).
 $$ 
 \end{Corollary} 
 
 \begin{proof}
 We have that $R[\mathcal I]\subset R[\overline{\mathcal I}]\subset\overline{R[\mathcal I]}$ by Lemma \ref{BdLemma1}. Thus 
 $s=s(\mathcal I)=s(\{\overline{I_n}\})=s(\{J_n\})$  and $e_s(\mathfrak a,\mathcal I)=e_s(\mathfrak a,\{\overline{I_n}\})=e_s(\mathfrak a,\{J_n\})$  by Theorem \ref{Rees1}.
 \end{proof}
 
The Minkowski inequalities were formulated and proven for   $m_R$-primary ideals in reduced equicharacteristic zero local rings by Teissier \cite{T1}, \cite{T2} and proven for  $m_R$-primary ideals in full generality, for  local rings,  by Rees and Sharp \cite{RS}. The same inequalities hold for filtrations. They were proven for $m_R$-filtrations in local rings $R$ such that $\dim N(\hat R)<\dim R$ in \cite[Theorem 6.3]{CSS}. We prove them for arbitrary filtrations in an analytically unramified local ring. 

\begin{Theorem}\label{Mink1} (Minkowski Inequalities)  Let $R$  be an analytically unramified  local ring of dimension $d,$   $N$ be a finitely generated $R$-module, $\mathfrak a$ be an $m_R$-primary ideal,      $\mathcal I=\{I_j\}$ and $\mathcal J=\{J_j\}$ be filtrations of ideals of $R.$ Let $\max\{s(\mathcal I), s(\mathcal J)\}\leq s< d$ and $k:=d-s$. 
	\begin{itemize}
		\item[$(i)$] Let $k\geq 2.$ For $1\le i\le k-1,$ $$e_s(\mathfrak a,\mathcal I^{[i]},\mathcal J^{[k-i]};N)^2\le e_s(\mathfrak a,\mathcal I^{[i+1]},\mathcal J^{[k-i-1]};N)e_s(\mathfrak a, \mathcal I^{[i-1]},\mathcal J^{[k-i+1]};N),$$
		\item[$(ii)$]  For $0\le i\le k$,	$e_s(\mathfrak a,\mathcal I^{[i]},\mathcal J^{[k-i]};N)e_s(\mathfrak a,\mathcal I^{[k-i]},\mathcal J^{[i]};N)\le e_{s}(\mathfrak a,\mathcal I;N)e_{s}(\mathfrak a,\mathcal J;N)$,\\
		\item[$(iii)$] For $0\le i\le k$, $e_s(\mathfrak a,\mathcal I^{[k-i]},\mathcal J^{[i]};N)^{k}\le e_{s}(\mathfrak a,\mathcal I;N)^{k-i}e_{s}(\mathfrak a,\mathcal J;N)^i$ and\\
		\item[$(iv)$]  $s(\mathcal I\mathcal J)=\max\{s(\mathcal I), s(\mathcal J)\}$ and $e_{s}(\mathfrak a,\mathcal I\mathcal J;N)^{\frac{1}{k}}\le e_{s}(\mathfrak a,\mathcal I;N)^{\frac{1}{k}}+e_{s}(\mathfrak a,\mathcal J;N)^{\frac{1}{k}}$, where $\mathcal I\mathcal J=\{I_jJ_j\}$.
	\end{itemize}
\end{Theorem}
\begin{proof}
	Let $\mfp\in\Spec R$ with $\dim R/\mfp=s$ and $\dim R_{\mfp}=d-s$. If $\mfp\in A({\mathcal I})$ then $I_{n}R_\mfp$ are $\mfp R_\mfp$-primary ideals for all $n\geq 1$ (respectively if $\mfp\in A({\mathcal J})$ then $J_{n}R_\mfp$ are $\mfp R_\mfp$-primary ideals for all $n\geq 1$). If $\mfp\notin A({\mathcal I})$ then $I_{n}R_\mfp=R_\mfp$ for all $n\geq 1$ (respectively if $\mfp\notin A({\mathcal J})$ then $J_{n}R_\mfp=R_\mfp$ for all $n\geq 1$).
	\\Let $T:=(A(\mathcal I)\cup A(\mathcal J))\cap\{\mathfrak p\in\mbox{Spec}(R): \dim R_{\mfp}=d-s\}$ and $|T|=r.$ Suppose $T=\{\mfp_1,\ldots,\mfp_r\}.$ 
	\\	 $(i)$ Suppose that $i$ satisfies $1\le i\le k-1$.
	Let $\mfp\in\Spec R$ with $\dim R/\mfp=s$ and $\dim R_{\mfp}=d-s$. 
	We  first show that
	\begin{equation}\label{eq50}
	e_{R_{\mfp}}(\mathcal I_{\mfp}^{[i]}, \mathcal J_{\mfp}^{[k-i]};N_{\mfp})^2\leq e_{R_{\mfp}}(\mathcal I_{\mfp}^{[i+1]},\mathcal J_{\mfp}^{[k-i-1]};N_{\mfp})e_{R_{\mfp}}(\mathcal I_{\mfp}^{[i-1]},\mathcal J_{\mfp}^{[k-i+1]};N_{\mfp}).
	\end{equation}	
	If $\mfp\in A({\mathcal I})\setminus A({\mathcal J})$ or $\mfp\in A({\mathcal J})\setminus A({\mathcal I})$,  in both cases we have
	$$e_{R_{\mfp}}(\mathcal I_{\mfp}^{[i]}, \mathcal J_{\mfp}^{[k-i]};N_{\mfp})=0\mbox{ for all }1\le i\le k-1.
	$$
	Suppose $\mfp\in A({\mathcal I})\cap A({\mathcal J}).$ Then by  \cite[Theorem 6.3, 1) , 2)]{CSS},  we have 
	$$e_{R_{\mfp}}(\mathcal I_{\mfp}^{[i]}, \mathcal J_{\mfp}^{[k-i]};N_{\mfp})^2\leq e_{R_{\mfp}}(\mathcal I_{\mfp}^{[i+1]},\mathcal J_{\mfp}^{[k-i-1]};N_{\mfp})e_{R_{\mfp}}(\mathcal I_{\mfp}^{[i-1]},\mathcal J_{\mfp}^{[k-i+1]};N_{\mfp})
	$$
	By (\ref{M9}), we have
	\begin{equation}\label{eq20} e_s(\mathfrak a,\mathcal I^{[d_1]},\mathcal J^{[d_2]};N)=\sum_{j=1 }^re_{R_{\mathfrak p_j}}(\mathcal I_{\mfp_j}^{[d_1]},\mathcal J_{\mfp_j}^{[d_2]};N_{\mfp_j})e_{\mathfrak a}(R/\mfp_j).
	\end{equation}
	For $1\le j\le r$, set 
	$$
	x_1(j)=e_{R_{\mathfrak p_j}}(\mathcal I_{\mfp_j}^{[i+1]},\mathcal J_{\mfp_j}^{[k-i-1]};N_{\mfp_j})^{\frac{1}{2}}e_{\mathfrak a}(R/\mfp_j)^{\frac{1}{2}}\mbox{ and }
	$$
	$$
	x_2(j)=e_{R_{\mathfrak p_j}}(\mathcal I_{\mfp_j}^{[i-1]},\mathcal J_{\mfp_j}^{[k-i+1]};N_{\mfp_j})^{\frac{1}{2}}e_{\mathfrak a}(R/\mfp_j)^{\frac{1}{2}}.	
	$$	
	Then, by (\ref{eq20}) and (\ref{eq50}), 
	\begin{eqnarray}\label{eq51}
	&&e_s(\mathfrak a,\mathcal I^{[i]},\mathcal J^{[k-i]};N)^2
	\\&=& (\sum_{j=1}^re_{R_{\mathfrak p_j}}(\mathcal I_{\mfp_j}^{[i]},\mathcal J_{\mfp_j}^{[k-i]};N_{\mfp_j})e_{\mathfrak a}(R/\mfp_j))^2\nonumber\\ 
	&\le& (\sum_{j=1}^rx_1(j)x_2(j))^2\le (\sum_{j=1}^rx_1(j)^2)(\sum_{j=1}^rx_2(j)^2)\nonumber\\
	&=&(\sum_{j=1}^re_{R_{\mathfrak p_j}}(\mathcal I_{\mfp_j}^{[i+1]},\mathcal J_{\mfp_j}^{[k-i-1]};N_{\mfp_j})e_{\mathfrak a}(R/\mfp_j))
	(\sum_{j=1}^re_{R_{\mathfrak p_j}}(\mathcal I_{\mfp_j}^{[i-1]},\mathcal J_{\mfp_j}^{[k-i+1]};N_{\mfp_j})e_{\mathfrak a}(R/\mfp_j))\nonumber\\
	&=& e_s(\mathfrak a,\mathcal I^{[i+1]},\mathcal J^{[k-i-1]};N)e_s(\mathfrak a, \mathcal I^{[i-1]},\mathcal J^{[k-i+1]};N)\nonumber,	
	\end{eqnarray}
	which establishes formula  $(i)$ of the statement of the theorem.
	The  inequality between the second and   third lines of (\ref{eq51}) is H\"older's inequality, formula (2.8.3) on page 24 of \cite{Ho}, with $k=2$ (so its conjugate $k'=\frac{k}{k-1}=2$ also).
	\\$(ii)$ and $(iii):$ 
Let $e_i=: e_s(\mathfrak a,\mathcal I^{[k-i]},\mathcal J^{[i]};N)$ for all $0\leq i\leq k.$ If $e_0=0$ or $e_k=0$ then by part $(i),$ we have $e_i=0$ for all $0<i<k$ and we get the inequalities $(ii)$ and $(iii).$ Suppose  $e_0>0$ and $e_k>0.$ If $e_i=0$ for some $i\in\{1,\ldots,k-1\}$ then using  part $(i),$ we get $e_i=0$ for all $0< i< k$ and hence the inequalities $(ii)$ and $(iii)$ hold. If $e_i>0$ for all  $0\leq i\leq k$ then the proof follows from the  argument given in \cite[Corollary 17.7.3, (1) and (2)]{HS}.	
\\$(iv)$ Since $V(I_1)\cup V(J_1)=V(I_1J_1),$ we have 
	$s(\mathcal I\mathcal J)=\max\{s(\mathcal I), s(\mathcal J)\}.$ Let $\mfp\in\Spec R$ with $\dim R/\mfp=s$ and $\dim R_{\mfp}=d-s$. 
	We first show that
	\begin{equation}\label{eq150}
	e_{R_{\mfp}}(\mathcal I_\mfp\mathcal J_\mfp; N_{\mfp})^{\frac{1}{k}}\le  e_{R_{\mfp}}(\mathcal I_\mfp; N_{\mfp})^{\frac{1}{k}}+ e_{R_{\mfp}}(\mathcal J_\mfp; N_{\mfp})^{\frac{1}{k}}.
	\end{equation}	
	\\If $\mfp\in A({\mathcal I})\setminus A({\mathcal J})$ (respectively $\mfp\in A({\mathcal J})\setminus A({\mathcal I})$) then 
	\begin{eqnarray*}&& e_{R_{\mfp}}(\mathcal I_\mfp\mathcal J_\mfp;N_{\mfp})^{\frac{1}{k}}= e_{R_{\mfp}}(\mathcal I_\mfp;N_{\mfp})^{\frac{1}{k}}\le e_{R_{\mfp}}(\mathcal I_\mfp;N_{\mfp})^{\frac{1}{k}}+e_{R_{\mfp}}(\mathcal J_\mfp;N_{\mfp})^{\frac{1}{k}}\\&&
		\hspace{-2.25cm}(\mbox{respectively } e_{R_{\mfp}}(\mathcal I_\mfp\mathcal J_\mfp;N_{\mfp})^{\frac{1}{k}}= e_{R_{\mfp}}(\mathcal J_\mfp;N_{\mfp})^{\frac{1}{k}}\le  e_{R_{\mfp}}(\mathcal I_\mfp;N_{\mfp})^{\frac{1}{k}}+ e_{R_{\mfp}}(\mathcal J_\mfp;N_{\mfp})^{\frac{1}{k}}).\end{eqnarray*}
	\\Suppose $\mfp\in A({\mathcal I})\cap A({\mathcal J}).$ Then by \cite[Theorem 6.3, 4)]{CSS},  we have $$ e_{R_{\mfp}}(\mathcal I_\mfp\mathcal J_\mfp; N_{\mfp})^{\frac{1}{k}}\le  e_{R_{\mfp}}(\mathcal I_\mfp; N_{\mfp})^{\frac{1}{k}}+ e_{R_{\mfp}}(\mathcal J_\mfp; N_{\mfp})^{\frac{1}{k}}$$ (where $\mathcal I_\mfp\mathcal J_\mfp=\{I_j J_jR_\mfp\}$). 
	By (\ref{M9}), we have
	\begin{equation}\label{eq120} e_{s}(\mathfrak a,\mathcal I\mathcal J;N)^{\frac{1}{k}} =\big(\sum_{j=1}^re_{R_{\mathfrak p_j}}(\mathcal I_{\mfp_j}\mathcal J_{\mfp_j};N_{\mfp_j})e_{\mathfrak a}(R/\mfp_j)\big)^{\frac{1}{k}}.
	\end{equation}
	For $1\le j\le r$, set 
	\begin{eqnarray}
	&& x_1(j)=e_{R_{\mfp_j}}(\mathcal I_{\mfp_j}; N_{\mfp_j})^{\frac{1}{k}}e_{\mathfrak a}(R/\mfp_j)^{\frac{1}{k}},
	~~x_2(j)=e_{R_{\mfp_j}}(\mathcal J_{\mfp_j}; N_{\mfp_j})^{\frac{1}{k}}e_{\mathfrak a}(R/\mfp_j)^{\frac{1}{k}}\mbox{ and }\nonumber\\&&
	u(j)=e_{R_{\mathfrak p_j}}(\mathcal I_{\mfp_j}\mathcal J_{\mfp_j};N_{\mfp_j})^{\frac{1}{k}}e_{\mathfrak a}(R/\mfp_j)^{\frac{1}{k}}.\nonumber\end{eqnarray}	
	Then $u(j)\leq x_1(j)+x_2(j)$ for all $j=1,\ldots,r$ and by (\ref{eq120}) and (\ref{eq150}), 
	\begin{eqnarray}\label{newmin}
	{\hspace{0.5cm}}~~e_{s}(\mathfrak a,\mathcal I\mathcal J;N)^{\frac{1}{k}} &=&\big(\sum_{j=1}^re_{R_{\mathfrak p_j}}(\mathcal I_{\mfp_j}\mathcal J_{\mfp_j};N_{\mfp_j})e_{\mathfrak a}(R/\mfp_j)\big)^{\frac{1}{k}}
	\\&=&\big(\sum_{j=1}^ru(j)^k\big)^{\frac{1}{k}}\leq \big(\sum_{j=1}^r(x_1(j)+x_2(j))^k\big)^{\frac{1}{k}}\nonumber
	\\&\leq& \big(\sum_{j=1}^r(x_1(j))^k\big)^{\frac{1}{k}} +\big(\sum_{j=1}^r(x_2(j))^k\big)^{\frac{1}{k}}\nonumber
	\\&=& \big(\sum_{j=1}^r(e_{R_{\mfp_j}}(\mathcal I_{\mfp_j}, N_{\mfp_j})e_{\mathfrak a}(R/\mfp_j)\big)^{\frac{1}{k}} + \big(\sum_{j=1}^r(e_{R_{\mfp_j}}(\mathcal J_{\mfp_j}; N_{\mfp_j})e_{\mathfrak a}(R/\mfp_j)\big)^{\frac{1}{k}}\nonumber\\&=&e_{s}(\mathfrak a,\mathcal I;N)^{\frac{1}{k}}+e_{s}(\mathfrak a,\mathcal J;N)^{\frac{1}{k}}\nonumber
	\end{eqnarray}		
	which establishes formula  $(iv)$ of the statement of the theorem. For $k>1,$ the inequality between  the second and third lines of (\ref{newmin}) is Minkowski's inequality, Section $2.12~(28)$ on page 32 of \cite{Ho}.	
\end{proof}

\begin{Theorem}\label{TheoremN10} Suppose that $R$ is an analytically unramified local ring of dimension $d$ and that $\mathfrak a$ is an $m_R$-primary ideal. 
	Let $\mathcal I(1)$ and $\mathcal I(2)$ be two filtrations of ideals on $R,$ $\max\{s(\mathcal I(1)), s(\mathcal I(1))\}\leq s$ and $d-s\geq 1.$ Suppose there exist $a,b\in \ZZ_{>0}$ such that $$\overline{\sum\limits_{n\geq 0}I(1)_{an}}=\overline{\sum\limits_{n\geq 0}I(2)_{bn}}.$$ 
	Then  the Minkowski equality 
	$$
	e_s(\mathfrak a,\mathcal I(1)\mathcal I(2))^{\frac{1}{d-s}}= e_0^{\frac{1}{d-s}}+e_{d-s}^{\frac{1}{d-s}}
	$$
	holds between $\mathcal I(1)$ and $\mathcal I(2)$ where $e_0=e_s(\mathfrak a,\mathcal I(1)^{[d-s]},\mathcal I(2)^{[0]})$ and $e_{d-s}=e_s(\mathfrak a,\mathcal I(1)^{[0]},\mathcal I(2)^{[d-s]})$.
\end{Theorem}
\begin{proof}
Since $\overline{\sum\limits_{n\geq 0}I(1)_{an}}=\overline{\sum\limits_{n\geq 0}I(2)_{bn}},$ by Lemmas \ref{dim}  and \ref{LemmaIC3} 
 we have $V(I(1)_{n})=V(I(2)_{n})$ for all $n\geq 1$ and $s(\mathcal I(1))=s(\mathcal I(2)).$ 

Let $\mathcal A(\mathcal I(1))=\mathcal A(\mathcal I(2))=\{\mfp_1,\ldots, \mfp_r\}.$
For all $n_1,n_2\in\mathbb N,$ let
$$
P(n_1,n_2)=\lim_{m\rightarrow\infty}\frac{e_s(\mathfrak a,R/I(1)_{mn_1}I(2)_{mn_2})}{m^{d-s}/(d-s)!}
$$
and for $1\le i\le r$, let
\begin{align*}
P_i(n_1,n_2)&=H_i(n_1,n_2)e_{\mathfrak a}(R/\mathfrak p_i)
\end{align*}
where $$H_i(n_1,n_2)=
\lim_{m\rightarrow\infty} \frac{\ell_{R_{\mfp_i}}(R_{\mathfrak p_i}/I(1)_{mn_1}R_{\mfp_i}I(2)_{mn_2}R_{\mfp_i})}{m^{d-s}/(d-s)!} .$$
	
For all $\mfp_i$, we have  $\overline{\sum\limits_{n\geq 0}I(1)_{an}R_{\mfp_i}}=\overline{\sum\limits_{n\geq 0}I(2)_{bn}R_{\mfp_i}}.$ Therefore by \cite[Theorem 8.4]{C3} and its proof,  which proves this theorem for $0$-filtrations, 
for all $1\le i\le r,$ there exists $c_i\in\mathbb R,$ such that for all $n_1,n_2\in\mathbb N,$ we have $$H_i(n_1,n_2)=c_i(\frac{n_1}{a}+\frac{n_2}{b})^{d-s}$$ and hence $$P_i(n_1,n_2)=c_i(\frac{n_1}{a}+\frac{n_2}{b})^{d-s}e_{\mathfrak a}(R/\mathfrak p_i).$$
Using Theorem \ref{Theorem5},
for all $n_1,n_2\in\mathbb N,$  we get 
\begin{equation}\label{eqgen}
P(n_1,n_2)=\sum\limits_{i=1}^rP_i(n_1,n_2)=\sum\limits_{i=1}^rc_i(\frac{n_1}{a}+\frac{n_2}{b})^{d-s}e_{\mathfrak a}(R/\mathfrak p_i)=c(\frac{n_1}{a}+\frac{n_2}{b})^{d-s}
\end{equation}
where $c=\sum\limits_{i=1}^rc_ie_{\mathfrak a}(R/\mathfrak p_i).$ Therefore $$P(1,1)^{\frac{1}{d-s}}=P(1,0)^{\frac{1}{d-s}}+P(0,1)^{\frac{1}{d-s}}.$$

\end{proof}

The following lemma is well known. We provide a proof for the convenience of the reader. Our proof is an outline of the proof in
 \cite[Lemma 14, page 8]{KG}.
 
 \begin{Lemma} \label{geq}
	Let $R$  be a $d$-dimensional  local ring and let $I_i \subseteq J_i$ be $m$-primary ideals for $i=1,\ldots,d.$ Then 
	$$e(I_1^{[1]},\ldots,I_d^{[1]}) \geq e(J_1^{[1]},\ldots,J_d^{[1]}).$$
\end{Lemma}


\begin{proof} It is enough to prove the statement when $I_1 = J_1,\ldots,I_{d-1} = J_{d-1}$, and $I_d \subseteq J_d$. If $d = 1$, then $\ell(R/I_1^n) \geq \ell(R/J_1^n)$, so necessarily $e(I_1) \geq e(J_1)$. Now let $d> 1.$ We may assume that $R$ has an infinite residue field. Then for a general  element $a \in I_1 = J_1,$ 
 $e(I_1^{[1]},\ldots,I_d^{[1]};R) = e(I_2^{[1]},\ldots,I_d^{[1]};R/(a))$ and $e(J_1^{[1]},\ldots,J_d^{[1]};R) = e(J_2^{[1]},\ldots,J_d^{[1]}; R/(a)).$ We are done by induction on $d.$
\end{proof}

	\begin{Definition}\label{trunc}
	Suppose that $\mathcal I=\{I_i\}$  is a filtration of ideals on a local ring $R$. Fix $a\in \ZZ_+.$ The $a$-th {\it truncated  filtration} $\mathcal I_a=\{I_{a,i}\}$ of $\mathcal I$ is defined  by 
	\[ I_{a,n}= \left\{
	\begin{array}{l l}
	~~~~~I_n & \quad \text{if $n\le a$ }\\ \vspace{0.3mm}\\
	\sum\limits_{\substack{i,j>0\\i+j=n}} I_{a,i}I_{a,j} & \quad \text{if $n>a.$ }
	\end{array} \right.\] 	
\end{Definition}

\begin{Lemma}\label{FiltIneq}
	Let $R$  be a $d$-dimensional  local ring, $\mathfrak a$ be an $m_R$-primary ideal, $\mathcal I(i)=\{I(i)_n\}$ and  $\mathcal J(i)=\{J(i)_n\}$ be filtrations of ideals on $R$  for $1\leq i\leq r$ with $J(i)_n\subset I(i)_n$ for all $1\leq i\leq r$  and $n\geq 1.$ Let $N$ be a finitely generated $R$-module.
	\begin{itemize}
	\item[$(i)$] Suppose $\dim N(\hat R)<d$ and $\mathcal I(i),$  $\mathcal J(i)$ are filtrations of $R$ by $m_R$-primary ideals.  Then $$e_R(\mathcal I(1)^{[d_1]},\ldots, \mathcal I(r)^{[d_r]};N)\leq e_R(\mathcal J(1)^{[d_1]},\ldots, \mathcal J(r)^{[d_r]};N).$$
	\item[$(ii)$] Suppose $R$ is analytically unramified and $\max\{s(\mathcal J(1)),\ldots,s(\mathcal J(r)\}\le s\le d.$ Then 
	$$e_s(\mathfrak a,\mathcal I(1)^{[d_1]},\ldots, \mathcal I(r)^{[d_r]};N)\leq e_s(\mathfrak a,\mathcal J(1)^{[d_1]},\ldots, \mathcal J(r)^{[d_r]};N).$$ 
	\end{itemize}
	\end{Lemma}
\begin{proof}
$(i)$ For all $a\in \ZZ_+$, consider the $a$-th truncated filtrations  $\mathcal I_a(i)=\{I_a(i)_m\}$  and  $\mathcal J_a(i)=\{J_a(i)_m\}$ for all $1\leq i\leq r$.  Given $a\in \ZZ_+$, there exists $f_a\in \ZZ_+$ such that $I_a(i)_{f_am}=(I_a(i)_{f_a})^m$ and $J_a(i)_{f_am}=(J_a(i)_{f_a})^m$ for all $m\ge 0$ and $i=1,\ldots,r$. Define  filtrations of $R$ by $m_R$-primary ideals by $\tilde{\mathcal I}_a(i)=\{I_a(i)_{f_am}\}$ and  $\tilde{\mathcal J}_a(i)=\{J_a(i)_{f_am}\}$ for all $1\leq i\leq r$. Then  by \cite[Proposition 6.2]{CSS}, \cite[Lemma 3.2]{CSS} and Lemma \ref{geq}, 
	\begin{eqnarray*}
		e_R(\mathcal I(1)^{[d_1]},\ldots, \mathcal I(r)^{[d_r]};N)&=&\lim\limits_{a\to\infty}e_R(\mathcal I_a(1)^{[d_1]},\ldots, \mathcal I_a(r)^{[d_r]};N)	\\&=& \lim\limits_{a\to\infty}\frac{1}{f_a^d}e_R(\tilde{\mathcal I}_a(1)_1^{[d_1]},\ldots,\tilde{\mathcal I}_a(r)_1^{[d_r]};N)\\&\leq & \lim\limits_{a\to\infty}\frac{1}{f_a^d}e_R(\tilde{\mathcal J}_a(1)_1^{[d_1]},\ldots,\tilde{\mathcal J}_a(r)_1^{[d_r]};N)\\&=&\lim\limits_{a\to\infty}e_R(\mathcal J_a(1)^{[d_1]},\ldots, \mathcal J_a(r)^{[d_r]};N)	\\&=& 	e_R(\mathcal J(1)^{[d_1]},\ldots, \mathcal J(r)^{[d_r]};N).
	\end{eqnarray*}
\\	$(ii)$ This follows from Theorem \ref{Theorem5} and part $(i)$ of this Lemma.
\end{proof}

  In \cite[Section 1]{Li} and \cite[Proposition 3.11]{HIO}, a multiplicity $e(\mathfrak a,I)$ is defined for ideals $\mathfrak a$ and $I$ in a local ring $R$ such that $\mathfrak a+I$ is $m_R$-primary, by
 \begin{equation}\label{Lieq1}
 e(\mathfrak a,I)=\sum_{\mfp}e_{R_{\mfp}}(I_{\mfp})e_{\mathfrak a}(R/\mfp),
\end{equation}
 where the sum is over prime ideals $\mfp$ in $R$ which contain $I$ and such that
 $$
 \dim R/\mfp=\dim R/I\mbox{ and }\dim R_{\mfp}=\dim R-\dim R/I.
 $$
 We generalize equation (\ref{Lieq1}) to filtrations. Suppose that $R$ is an analytically unramified local ring of dimension $d$ and $\mathcal I$ is a filtration of ideals on $R.$ Let $s=s(\mathcal I).$ Suppose  $\mathfrak a$ is an ideal in $R$ such that  $\mathfrak a+I_1$ is $\mR$-primary. 
 We define
 	\begin{equation}\label{Lieq2}
	e(\mathfrak a, \mathcal I)=\sum_{\mathfrak p}e_{R_{\mathfrak p}}(\mathcal I_{\mathfrak p})e_{\mathfrak a}(R/\mathfrak p),
	\end{equation}
	 where the sum is over all $\mathfrak p\in{\Spec(R)}$ such that $\dim R/\mathfrak p=s$ and $\dim R_{\mathfrak p}=d-s$.
	 
We have the following formula relating equations (\ref{Lieq1}) and (\ref{Lieq2}):

\begin{equation}\label{Lieq3}
e(\mathfrak a,\mathcal I)=
\lim_{m\to\infty}\frac{e(\mathfrak a,  I_m)}{m^{d-s}}.
 	\end{equation}
To prove this formula, we observe that 
for all $m,n \geq 1,$ $$\Min(R/I_n)\cap\{p\in{\Spec(R): \dim R/\mathfrak p=s\mbox{ and }\dim R_{\mathfrak p}=d-s}\}$$ $$=\Min(R/I_m)\cap\{p\in{\Spec(R): \dim R/\mathfrak p=s\mbox{ and }\dim R_{\mathfrak p}=d-s}\},$$ so that the limit 
 	\begin{align*}\lim_{m\to\infty}\frac{e(\mathfrak a,  I_m)}{m^{d-s}}&=\sum_{\mathfrak p}\lim\limits_{m\to\infty}\frac{e_{R_{\mathfrak p}}( I_mR_{\mathfrak p})}{m^{d-s}}e_{\mathfrak a}(R/\mathfrak p)
 	\\&=\sum_{\mathfrak p}\lim\limits_{n\to\infty}\frac{\ell_{R_{\mfp}}(R_\mfp/I_nR_\mfp)}{n^{d-s}/(d-s)!}e_{\mathfrak a}(R/\mathfrak p)\\&=\sum_{\mathfrak p}e_{R_{\mathfrak p}}(\mathcal I_{\mathfrak p})e_{\mathfrak a}(R/\mathfrak p)
	 	\end{align*} exists, where the  second equality follows from \cite[Theorem 6.5]{C1} and the sum is over all $\mathfrak p\in{\Spec(R)}$ such that $\dim R/\mathfrak p=s$ and $\dim R_{\mathfrak p}=d-s.$

 Let $\mathcal I$ be a filtration of ideals on a local ring $R$. We say $x_1,\ldots,x_{s(\mathcal I)}$ is a system of parameters of $\mathcal I$ if $x_1+I_n,\ldots,x_{s(\mathcal I)}+I_n$ is a system of parameters of $R/I_n$ for each $n\geq 1.$ Note that if $x_1,\ldots,x_t$ is a part of system of parameters of $\mathcal I$ then $\{I_n+(x_1,\ldots,x_t)\}$
 is a filtration and thus $V(I_m+(x_1,\ldots,x_t))=V(I_n+(x_1,\ldots,x_t))$ for all $m,n\geq 1.$
 \begin{Remark}
 	Let $R$ be a Regular local ring of dimension $d\geq 1,$  $\mathcal I$ be a filtration of ideals on $R$ which is not an $\mR$-filtration. Then using the prime avoidance lemma, we can choose elements $x_i\in\mR\setminus \mR^2\bigcup \Min(R/I_1+(x_1,\ldots,x_{i-1})) $ for all $1\leq i\leq s(\mathcal I)$ and $x_0=0.$ Then $x_1,\ldots,x_{s(\mathcal I)}$ is a system of parameters of $\mathcal I$ such that $x_1,\ldots,x_{s(\mathcal I)}$ is a part of regular system of parameters of $R.$ Define $\underline x= x_1,\ldots,x_{s(\mathcal I)}$. Therefore $R/\underline xR$ is a regular local ring and hence $e_{R/\underline xR}(\mathcal J)$ is well-defined where $\mathcal J=\{J_n=I_n(R/\underline xR)\}.$
 \end{Remark}

  In \cite{HSV} and \cite[formula (2.1), page 118]{Li}, the following inequality is given. Suppose that $I$ is an ideal in a local ring $R$. Let $s=\dim R/I$. Suppose that $\underline x= x_1,\ldots,x_s$ is part of a  system of parameters in $R$ and that the image of $\underline x$  in $R/I$ is a system of parameters. Then 
 \begin{equation}\label{eqLie4}
 e(\underline xR,I)\le e_{R/\underline x R}(I).
 \end{equation}
 In the following proposition, we generalize this inequality to filtrations.

 \begin{Proposition}\label{PropUS}
 	Let $R$ be an analytically unramified local ring of dimension $d\geq 1,$  $\mathcal I$ be a filtration of ideals on $R$ which is not an $\mR$-filtration and $x_1,\ldots,x_{s(\mathcal I)}$ is a system of parameters of $\mathcal I$ such that $\dim(N(\widehat{R/\underline xR}))<\dim R/\underline xR$ (e.g. $R$ is a Regular local ring and $x_1,\ldots,x_{s(\mathcal I)}$ is a system of parameters of $\mathcal I$ such that $x_1,\ldots,x_{s(\mathcal I)}$ is a part of regular system of parameters of $R$) where $N(\widehat{R/\underline xR})$ is the nilradical of $R/\underline xR.$ Let ${\underline x}=x_1,\ldots,x_{s(\mathcal I)}.$ Then for  $\mathcal J=\{J_n=I_n(R/\underline xR)\},$ 
 	$$ e(\underline xR, \mathcal I)\leq e_{R/\underline xR}(\mathcal J).$$
 \end{Proposition}
 
 \begin{proof}
 	Let $\mathcal I_a$ denote the $a$-th truncated filtration of $\mathcal I$ for all $a\geq 1$  and $\mfp$ be a prime ideal in $R$. Then $\mathcal I_a(R/\underline xR)=\{I_{a,n}(R/\underline xR)\}$ and $\mathcal I_aR_\mfp=\{I_{a,n}R_\mfp\}$ are the  $a$-th truncated filtrations of $\mathcal I_a(R/\underline xR)$ and $\mathcal I_aR_\mfp$ respectively for all $a\geq 1$.  Since for each $a\geq 1,$ $\mathcal I_a,$ $\mathcal I_a(R/\underline xR)$  and $\mathcal I_aR_\mfp$ are Noetherian filtrations there exists an integer $f_a\geq 1$ such that $I_{a,nf_a}= I_{a,f_a}^n,$ $I_{a,nf_a}(R/\underline xR)= I_{a,f_a}^n(R/\underline xR)$ and $I_{a,nf_a}R_\mfp= I_{a,f_a}^nR_\mfp$ for all $n\geq 1.$ Also note that $s(\mathcal I)=s(\{I_{a,f_a}^n\}).$
 	Then  summing over  prime ideals $\mfp$ with $\dim R/\mathfrak p=s(\mathcal I)$ and $\dim R_{\mathfrak p}=d-s(\mathcal I),$ we have 
 	\begin{align*}
 	 e(\underline xR, \mathcal I)=\sum_{\mathfrak p}e_{R_{\mathfrak p}}(\mathcal I_{\mathfrak p})e_{\underline xR}(R/\mathfrak p)&= \sum_{\mathfrak p}\lim_{a\to \infty}  e_{R_{\mathfrak p}}(\mathcal I_aR_{\mathfrak p})e_{\underline xR}(R/\mathfrak p)\\&= \sum_{\mathfrak p}\lim_{a\to \infty} \frac{1}{f_a^{d-s(\mathcal I)}} e_{R_{\mathfrak p}}( I_{a,f_a}R_{\mathfrak p})e_{\underline xR}(R/\mathfrak p)	\\&= \lim_{a\to \infty} \frac{1}{f_a^{d-s(\mathcal I)}}\sum_{\mathfrak p} e_{R_{\mathfrak p}}( I_{a,f_a}R_{\mathfrak p})e_{\underline xR}(R/\mathfrak p)\\&= \lim_{a\to \infty} \frac{1}{f_a^{d-s(\mathcal I)}} e(\underline xR, I_{a,f_a})\\&\leq  \lim_{a\to \infty} \frac{1}{f_a^{d-s(\mathcal I)}}e( I_{a,f_a}(R/\underline xR))\nonumber
 	\\&=\lim_{a\to \infty}e( \mathcal I_a(R/\underline xR))=e( \mathcal I(R/\underline xR))\nonumber
 	\end{align*} where  the equalities on the first and last lines are by \cite[Proposition 6.2]{CSS} and the  inequality follows from (\ref{eqLie4}). \end{proof}

\section{divisorial filtrations}\label{SecDiv}
Let $R$ be a  local domain of dimension $d$ with quotient field $K$.  Let $\nu$ be a discrete valuation of $K$ with valuation ring $\mathcal O_{\nu}$ and maximal ideal $m_{\nu}$.  Suppose that $R\subset \mathcal O_{\nu}$. Then for $n\in \NN$, define valuation ideals
$$
I(\nu)_n=\{f\in R\mid \nu(f)\ge n\}=m_{\nu}^n\cap R.
$$
 
  A divisorial valuation of $R$ (\cite[Definition 9.3.1]{HS}) is a valuation $\nu$ of $K$ such that if $\mathcal O_{\nu}$ is the valuation ring of $\nu$ with maximal ideal $\mathfrak m_{\nu}$, then $R\subset V_{\nu}$ and if $\mfp=\mathfrak m_{\nu}\cap R$ then $\mbox{trdeg}_{\kappa(\mfp)}\kappa(\nu)={\rm ht}(\mfp)-1$, where $\kappa(\mfp)$ is the residue field of $R_{\mfp}$ and $\kappa(\nu)$ is the residue field of $V_{\nu}$.  
 
 By \cite[Theorem 9.3.2]{HS}, the valuation ring of every divisorial valuation $\nu$ is Noetherian, hence is a  discrete valuation. 
 
 \begin{Lemma}
Suppose that  $R$ is an excellent local domain. Then a valuation $\nu$ of the quotient field $K$ of $R$ which is nonnegative on $R$ is a divisorial valuation of $R$ if and only if the valuation ring $\mathcal O_{\nu}$  is essentially of finite type over $R$. 
\end{Lemma}

\begin{proof} 
Since an excellent local domain is analytically unramified, the only if direction follows from \cite[Theorem 9.3.2]{HS}.
Now we establish the if direction. Since $\mathcal O_{\nu}$ is essentially of finite type over $R$, there exists a finite type $R$-algebra $S$ and a prime ideal $Q$ in $S$ such that  $S$ is a sub $R$-algebra of $\mathcal O_{\nu}$ and $S_Q=\mathcal O_{\nu}$.  Since an excellent local domain is universally catenary, the  dimension equality (c.f. \cite[Theorem B.3.2.]{HS}) holds. Since a Noetherian valuation ring is a discrete valuation ring (c.f. \cite[Corollary 6.4.5]{HS}) it has dimension 1, so that ${\rm ht}(Q)=1$, from which it follows that  $\nu$ is a divisorial valuation.
\end{proof}


 Suppose that $s\in \NN$.
An $s$-valuation of $R$ is a divisorial valuation of $R$ such that $\dim R/p=s$ where $p=m_{\nu}\cap R$.

A divisorial filtration of $R$ is a filtration $\mathcal I=\{I_m\}$ such that  there exist divisorial valuations $\nu_1,\ldots,\nu_r$ and $a_1,\ldots,a_r\in \RR_{\ge 0}$ such that for all $m\in \NN$,
$$
I_m=I(\nu_1)_{\lceil ma_1\rceil }\cap\cdots\cap I(\nu_r)_{\lceil ma_r\rceil}.
$$
A divisorial filtration is called integral (rational) if $a_i\in  \ZZ_{\ge 0}$ for all $i$ ($a_i\in \QQ_{\ge 0}$ for all $i$).

An $s$-divisorial filtration of $R$ is a filtration $\mathcal I=\{I_m\}$ such that  there exist $s$-valuations $\nu_1,\ldots,\nu_r$ and $a_1,\ldots,a_r\in \RR_{\ge 0}$ such that for all $m\in \NN$,
\begin{equation}\label{eqDF}
I_m=I(\nu_1)_{\lceil ma_1\rceil }\cap\cdots\cap I(\nu_r)_{\lceil ma_r\rceil}.
\end{equation}

 Observe that the trivial filtration $\mathcal I=\{I_m\}$, defined by $I_m=R$ for all $m$, is a degenerate case of a divisorial filtration and is a degenerate case of an $s$-divisorial filtration for all $s$.
The  nontrivial $0$-divisorial filtrations are the  divisorial $m_R$-filtrations of \cite{C3}.

 We will often denote a divisorial filtration $\mathcal I$ on a local domain $R$ by $\mathcal I=\mathcal I(D)$, even when $R$ is not excellent and there does not exist a representation of  $\mathcal I$ as defined before Theorem \ref{Theorem9}.

 Even when that $a_i$ are required to be positive for all $i$, the expression (\ref{eqDF}) of a divisorial filtration is far from unique. The following example follows from \cite[Theorem 4.1]{C5}.


\begin{Example} There exist $0$-valuations $\nu_1$ and $\nu_2$ on a normal 
 3-dimensional local ring $R$ which is essentially of finite type over an arbitrary algebraically closed field $k$, such that if $a_1,a_2,b_1,b_2\in \NN$ and 
$a_1<\left(\frac{3}{9-\sqrt{3}}\right)a_2$, then 
$$
I(\nu_1)_{\lceil ma_1\rceil}\cap I(\nu_2)_{\lceil ma_2\rceil}=I(\nu_1)_{\lceil mb_1\rceil}\cap I(\nu_2)_{\lceil mb_2\rceil}
$$
for all $m\in \NN$ if and only if $b_2=a_2$ and $b_1<\left(\frac{3}{9-\sqrt{3}}\right)a_2$. The filtration has the largest expression 
as a real divisorial filtration  as $\left\{I(\nu_1)_{\lceil m\left(\frac{3}{9-\sqrt{3}}\right)a_2\rceil}\cap I(\nu_2)_{\lceil ma_2\rceil}\right\}$.  This divisorial filtration is not rational.
\end{Example}

\begin{Lemma}\label{LemmaICD}
If $\mathcal I(D)$ is a divisorial filtration, then $R[\mathcal I(D)]=\overline{R[\mathcal I(D)]}$ is integrally closed. 
\end{Lemma}

The proof of Lemma \ref{LemmaICD}  for $m_R$-filtrations in \cite[Lemma 5.7]{C3} extends immediately to arbitrary divisorial filtrations. 

We will use the following form of the   valuative criterion of properness. The proposition is an immediate   consequence of  \cite[Theorem II.4.7]{Ha}.

\begin{Proposition}\label{valcri} Suppose that $R$ is a Noetherian  domain with quotient field $K$ and that $I$ is a nonzero  ideal of $R$. Let $\pi:X\rightarrow \mbox{Spec}(R)$ be the blow up of $I$. Suppose that $\mathcal O_{\nu}$ is a valuation ring of $K$ such that $R\subset \mathcal O_{\nu}$. Let $\mathfrak p=m_{\nu}\cap R$. Then there exists a unique (not necessarily closed) point $\alpha$ of $X$ such that $\mathcal O_{\nu}$ dominates $S=\mathcal O_{X,\alpha}$. We have that $S$ dominates $R_{\mathfrak p}$; that is, $\pi(\alpha)=\mathfrak p\in \mbox{Spec}(R)$.
\end{Proposition}

\begin{Lemma}\label{Lemma0} Let $R$ be an excellent local domain and $\nu_1,\ldots,\nu_t$ be $s$-valuations of $R$ with associated centers $\mathfrak p_i=m_{\nu_i}\cap R$ on $R$  for all $1\leq i\leq t$. Then there exists an ideal $K$ of $R$ such that the associated primes of $R$ are the $\mathfrak p_i$ and if $\phi:X\rightarrow \mbox{Spec}(R)$ is the normalization of the blowup of $K$, then there exist prime Weil divisors $E_1,\ldots,E_t$ on $X$ such that $\mathcal O_{X,E_i}=\mathcal O_{\nu_i}$ for $1\le i\le t$.
\end{Lemma}

\begin{proof} After reindexing the $\mathfrak p_i$ we may suppose that $\mathfrak p_1,\ldots,\mathfrak p_r$  (with $r\le t$) are the distinct primes in the set $\{\mathfrak p_1,\ldots,\mathfrak p_t\}$. Then reindex the $\nu_i$ as $\nu_{i,j}$ for $1\le j\le \beta_i$ so that $m_{\nu_{i,j}}\cap R=\mathfrak p_i$ for all $i,j$.

It follows from Statement (G) of \cite{R3} (or \cite[Proposition 10.4.4]{HS}) that there exists a $(\mathfrak p_i)_{\mathfrak p_i}$-primary ideal $J_{i,j}$ of $R_{\mathfrak p_i}$ such that if $\phi_{i,j}:X_{i,j}\rightarrow \mbox{Spec}(R_{\mathfrak p_i})$ is the normalization of the blowup of $J_{i,j}$, then $\mathcal O_{\nu_{i,j}}$ is a local ring 
of a prime Weil divisor  on $X_{i,j}$. Let $J_i=\prod_jJ_{i,j}$ and $\phi_i:X_i\rightarrow \mbox{Spec}(R_{\mfp_i})$  be the normalization of the blow up of $J_i$.
 
 We will now show that $\mathcal O_{\nu_{i,j}}$ is a local ring 
 of  $X_{i}$ for all $i$.  
 Let $Y_{i,j}$ be the blowup of $J_{i,j}$ and $Y_i$ be the blowup of $J_i$, so that there are natural finite birational projective  morphisms $X_{i,j}\rightarrow Y_{i,j}$ and $X_i\rightarrow Y_i$. Let $J_{i,j}=(a_{i,j,1},\ldots,a_{i,j,\alpha_{i,j}})$.   
 The ideal sheaves $J_{i,j}\mathcal O_{Y_i}$ are locally principal, since $Y_i$ is covered by the open affine sets with 
 the affine coordinate rings 
 $T_{i,k_1,\ldots,k_{\beta_i}}=R_{\mathfrak p_i}[J_i/f]$ where $f=a_{i,1,k_1}\cdots a_{i,\beta_i,k_{\beta_i}}$ for some $k_1,\ldots,k_{\beta_i}$.
  The ideal sheaves $J_{i,j}\mathcal O_{X_i}$ are thus locally principal,  since 
 $X_i$ is covered by the open affine sets 
 with the affine coordinate rings $S_{i,k_1,\ldots,k_{\beta_i}}$ where $S_{i,k_1,\ldots,k_{\beta_i}}$ is the normalization of  $T_{i,k_1,\ldots,k_{\beta_i}}$.  Thus we have a  birational projective morphism $X_{i}\rightarrow Y_{i,j}$ for all $i,j$ (by the universal property of blowing up, c.f. \cite[Proposition II.7.14]{Ha}). Since $X_{i}$ is normal, there is a birational projective morphism
 $X_i\rightarrow X_{i,j}$ for all $i,j$. Now for fixed $i,j$, by Proposition \ref{valcri}, there exists a unique local ring $A$ of $X_i$ such that the valuation ring $\mathcal O_{\nu_{i,j}}$ dominates $A$. 
 
 
 Let $B$ be the local ring of $X_{i,j}$ such that $A$ dominates $B$. Now $B$ is the unique local ring of $X_{i,j}$ which is dominated by $\nu_{i,j}$  by Proposition \ref{valcri}.
 Since $\mathcal O_{\nu_{i,j}}$ is a local ring of $X_{i,j}$ we have that $B=\mathcal O_{\nu_{i,j}}$ so that $\mathcal O_{\nu_{i,j}}=A$ is a local ring of $X_i$.

 For each $i$, there exists a $\mathfrak p_i$-primary ideal $K_i$ of $R$ such that $(K_i)_{\mathfrak p_i}=J_i$. Let $K=\cap_{i=1}^rK_i$. The associated primes of $K$ are thus $\mathfrak p_1,\ldots,\mathfrak p_r$. Let $\phi:X\rightarrow \mbox{Spec}(R)$ be the normalization of the blowup of $K$.
 
For given $i,j$, there exists $f_j\in J_i$ such that $\mathcal O_{\nu_{ij}}$ is a local ring of the normalization $T_{i,j}$ of $R_{\mathfrak p_i}[J_i/f_j]$.  Since $J_i=K_{\mathfrak p_i}$, we  may assume that $f_j\in K$, so that $R_{\mathfrak p_i}[J_i/f_j] 
 =R[K/f_j]_{\mathfrak p_i}$. Let $U_{i,j}$ be the normalization of $R[K/f_j]$, so that $U_{i,j}$ is the affine coordinate ring of an open subset of $X$. Now $T_{i,j}$ is the integral closure of $R[K/f_j]_{\mathfrak p_j}= R_{\mathfrak p_i}[J_i/f_j]$, so that  $(U_{i,j})_{\mathfrak p_i}=(T_{i,j})_{\mathfrak p_i}$, and so $\mathcal O_{\nu_{i,j}}$ is a local ring of $U_{i,j}$, and hence is a local ring of $X$.
   \end{proof}
   
   \begin{Remark} If $R$ is an excellent local domain and $\nu_1,\ldots,\nu_t$ are divisorial valuations of $R$, then a slight modification of the above proof gives the weaker statement that there exists an ideal $K$ of $R$ such that if $\phi:X\rightarrow\mbox{Spec}(R)$  is the normalization of the blowup of $K$, then there exist prime Weil divisors $E_1,\ldots,E_t$ on $X$ such that $\mathcal O_{X,E_i}=\mathcal O_{\nu_i}$ for $1\le i\le t$.
   \end{Remark}

  As in  \cite[Chapter 5]{C3}, Lemma \ref{Lemma0} allows us to define a representation of an $s$-divisorial filtration $\mathcal I=\{I_m\}$ on an excllent local domain $R$, where 
 $$
I_m=I(\nu_1)_{\lceil ma_1\rceil }\cap\cdots\cap I(\nu_t)_{\lceil ma_t\rceil}.
$$  
By Lemma \ref{Lemma0}, we may construct a blow up $\phi:X\rightarrow \mbox{Spec}(R)$ satisfying the conclusions of the lemma for $\nu_1,\ldots,\nu_t$. Let $D$ be the real Weil divisor $D=a_1E_1+\cdots+a_tE_t$ on $X$. Define
$$
I(mD)=\Gamma(X,\mathcal O_X(-\lceil \sum ma_iE_i\rceil)\cap R.
$$
giving a filtration $\mathcal I(D)=\{I(mD)\}$ on $R$. We have that 
$$
I(mD)=\Gamma(X,\mathcal O_X(-\lceil \sum ma_iE_i\rceil)\cap R
=I(\nu_1)_{\lceil ma_1\rceil }\cap\cdots\cap I(\nu_t)_{\lceil ma_t\rceil}=I_m
$$
so that $I(mD)$ is the $s$-divisorial filtration $\mathcal I$. 

Given a representation $\mathcal I(mD)$ of an $s$-divisorial filtration $\mathcal I$, it is desirable sometimes to separate the prime components of $D$ which dominate different prime ideals of $R$, as in the proof of Lemma \ref{Lemma0}.

Let $\mathfrak p_1,\ldots,\mathfrak p_r$ be the distinct prime ideals of $R$ which are dominated by a prime component of $D$.  Reindex the $E_i$ as $E_{i,j}$, where for each $i$, $\{E_{i,j}\}$ are the prime divisors (from the set $\{E_i\}$) which dominate $\mathfrak p_i$.
Let $a_{ij}\in \RR_{>0}$ be defined so that 
$D=\sum_{i,j} a_{i,j}E_{i,j}$. Define $D(i)=\sum_ja_{i,j}E_{i,j}$ for $1\le i\le r$. Then the filtrations $\mathcal I(D(i))=\{I(mD(i))\}$ are $s$-divisorial filtrations on $R$ and 
$I(mD)=I(mD(1))\cap\cdots \cap I(mD(r))$ for all $m\ge 0$.

\begin{Theorem}\label{Theorem9} Suppose that $R$ is an excellent local domain and $\mathfrak a$ is an $m_R$-primary ideal.
Let $\mathcal I(D)$  be a real $s$-divisorial filtration and $\mathcal I$ be an arbitrary filtration.  Suppose 
that $\mathcal I(D)\subset \mathcal I$, so that $s(\mathcal I)\le s=s(\mathcal I(D))$.  Then   $e_s(\mathfrak a,\mathcal I(D))=e_s(\mathfrak a,\mathcal I)$ if and only if
$I(mD)=I_m$ for all $m\ge 0$.
\end{Theorem}

\begin{proof} Let $\mathcal I=\{I_m\}$. First suppose that $I(mD)=I_m$ for all $m\ge 0$. Then $e_s(\mathfrak a,\mathcal I(D))=e_s(\mathfrak a,\mathcal I)$ by definition of $e_s$. 

Now suppose that $e_s(\mathfrak a,\mathcal I(D))=e_s(\mathfrak a,\mathcal I)$.
Let $\phi:X\rightarrow \mbox{Spec}(R)$ be a representation of the filtration $\mathcal I(D)$. There are prime ideals $\mathfrak p_1,\ldots,\mathfrak p_r$ in $R$ such that $\dim R/\mathfrak p_i=s$ for all $i$
and $X$ is  the normalization of the blowup of an ideal $K$ of $R$ such that $\mathfrak p_1,\ldots,\mathfrak p_r$ are the associated primes of $K$. Further,  $D=\sum_{i=1}^rD(i)$ is an   effective Weil divisor on $X$ such that for all $i$,  all prime components $E$ of $D(i)$ satisfy $m_E\cap R=\mathfrak p_i$, where $m_E$ is the maximal ideal of $\mathcal O_{X,E}$. 

Since excellent  rings are universally catenary, we have by Proposition \ref{Prop2} that 
$$
e_s(\mathfrak a,\mathcal I(D))=\sum_{i=1}^re_{R_{\mathfrak p_i}}(\mathcal I(D)_{\mathfrak p_i})e_{\mathfrak a}(R/\mathfrak p_i)
$$
and  since  $V(I_1)\subset V(I(D))$, we have 
$$
e_s(\mathfrak a,\mathcal I) =\sum_{i=1}^re_{R_{\mathfrak p_i}}(\mathcal I_{\mathfrak p_i})e_{\mathfrak a}(R/\mathfrak p_i).
$$
 Now for all $i$,  we have that $e_{R_{\mathfrak p_i}}(\mathcal I(D)_{\mathfrak p_i})\ge e_{R_{\mathfrak p_i}}(\mathcal I_{\mathfrak p_i})$ since $\mathcal I(D)_{\mfp_i}\subset \mathcal I_{\mfp_i}$. Thus $e_s(\mathfrak a,\mathcal I(D))\ge e_s(\mathfrak a,\mathcal I)$ with equality if and only if $e_{R_{\mathfrak p_i}}(\mathcal I(D)_{\mathfrak p_i})= e_{R_{\mathfrak p_i}}(\mathcal I_{\mathfrak p_i})$ for all $i$.

Suppose  that  for some $i$,  $\mathcal I_{\mathfrak p_i}$ is a filtration of $m_{R_{\mathfrak p_i}}$-primary ideals. By \cite[Theorem 1.4]{C4} and \cite[Corollary 7.5]{C3}, 
$e_{R_{p_i}}(\mathcal I(D)_{\mfp_i})=e_{R_{\mfp_i}}(\mathcal I_{\mfp_i})$ if and only if $\mathcal I(D)_{\mfp_i}=\mathcal I_{\mfp_i}$.
Also, if $\mathcal I_{\mathfrak p_i}$ is the trivial filtration, then since 
$\mathcal I(D)_{p_i}$ is a filtration of $m_{R_{\mathfrak p_i}}$-primary ideals,
we have that   $e_{R_{\mathfrak p_i}}(\mathcal I(D)_{\mathfrak p_i})>0$ by \cite[Proposition 5.3]{C3} while $e_{R_{\mathfrak p_i}}(\mathcal I_{\mathfrak p_i})=0$, so that $e_{R_{\mathfrak p_i}}(\mathcal I_{\mathfrak p_i})<e_{R_{\mathfrak p_i}}(\mathcal I(D)_{\mathfrak p_i})$. Thus 
$e_s(\mathfrak a,\mathcal I(D))=e_s(\mathfrak a,\mathcal I)$ if and only if $I(mD)_{\mathfrak p_i}=(I_m)_{\mathfrak p_i}$ for all $i$ and $m\in \NN$. 
In particular, with our assumption that $e_s(\mathfrak a,\mathcal I(D))=e_s(\mathfrak a,\mathcal I)$, we have that
\begin{equation}\label{eqR3}
I(mD)_{\mathfrak p_i}=(I_m)_{\mathfrak p_i}\mbox{ for all $i$ and $m\in \NN$}. 
\end{equation}

Now each $I(mD)$ is an intersection $Q_1\cap \cdots \cap Q_r$ where each $Q_i$ is a $\mfp_i$-primary ideal. Thus $I_m=Q_1\cap\cdots\cap Q_r\cap J$  where $J$  is an ideal of $R$ such that $J\not\subset \mfp_i$  for any $i$. But then 
$I_m=I(mD)$ since $I(mD)\subset I_m$.
 \end{proof}

	\begin{Lemma}\label{LemnMineq}  Suppose that $R$ is a $d$-dimensional  analytically unramified local ring, $\mathfrak a$ is an $m_R$-primary ideal    and $\mathcal I(1)=\{I(1)_j\}$ and $\mathcal I(2)=\{I(2)_j\}$ are filtrations such that $s(\mathcal I(1))=s(\mathcal I(2))$. Let 
		$s=s(\mathcal I(1))=s(\mathcal I(2))$ and let $e_i=e_s(\mathfrak a,\mathcal I(1)^{[d-s-i]},\mathcal I(2)^{[i]})$ for $0\le i\le d-s$.
		Then either
		\begin{enumerate}
			\item[a)] $e_0=0$ or $e_{d-s}=0$ and $e_i=0$ for $0<i<d-s$ or
			\item[b)] $e_0>0$ and $e_{d-s}>0.$
		\end{enumerate}

		For all $n_1,n_2\in\mathbb N,$ let 
		$$
		P_s(n_1,n_2)=\lim_{m\rightarrow\infty}\frac{e_s(\mathfrak a,R/I(1)_{mn_1}I(2)_{mn_2})}{m^{d-s}/(d-s)!}
		=\sum_{j=0}^{d-s} \frac{(d-s)!}{(d-s-j)!j!}e_jn_1^{d-s-j}n_2^j.
		$$
		
		Consider the following conditions.
		\begin{enumerate}
			\item[1)] $e_i^2= e_{i-1}e_{i+1}$ 
			for $1 \le i\le d-s-1$.
			\item[2)]   
			$e_ie_{d-s-i}= e_0e_{d-s}$ for $0\le i\le d-s$.
			\item[3)] $e_i^{d-s}= e_0^{d-s-i}e_{d-s}^i$ for $0\le i\le d-s$. 
			\item[4)]  $e_s(\mathfrak a,\mathcal I(1)\mathcal I(2)))^{\frac{1}{d-s}}= e_0^{\frac{1}{d-s}}+e_{d-s}^{\frac{1}{d-s}}$,  where $\mathcal I(1)\mathcal I(2)=\{I(1)_jI(2)_j\}$.
			\item[5)] $P_s(n_1,n_2)=\left(e_0^{\frac{1}{d-s}}n_1+e_{d-s}^{\frac{1}{d-s}}n_2\right)^{d-s}$ for all $n_1,n_2\in\mathbb N.$ 
		\end{enumerate}
		Then the following hold.
		\begin{enumerate}
		\item[i)] Statement 3) holds if and only if  statement 4) holds. 
		\item[ii)] If a) holds then  the statements 1) - 5) hold. 
		\item[iii)] If the $e_i$ are nonzero for  $0\le i\le d-s$, then  the statements 1) - 5) are equivalent. 
			\end{enumerate}
	\end{Lemma}

\begin{proof}
We have that either a) or b) holds by the   Minkowski inequalities for filtrations (Theorem \ref{Mink1}).

$(i)$ The analysis of \cite[Section 9]{C3} is valid for arbitrary filtrations, and shows that Statement 3) holds if and only if  statement 4) holds.

$ii)$ If a) holds, then all of the equalities 1) - 5) hold. 

$iii)$ Suppose that all $e_i$ are nonzero. We show that all of the equalities 1) - 5) hold.  The proof of  \cite[Section 9]{C3} applies here to show that  conditions 1), 3), 4) and 5) are equivalent. 

It remains to show that 2) is equivalent to 3). By the inequality iii) of the Minkowski inequalities for filtrations (Theorem \ref{Mink1}) we have that
\begin{equation}\label{eqM1}
e_i^{d-s}e_{d-s-i}^{d-s}\le (e_0^{d-s-i}e_{d-s}^i)(e_0^ie_{d-s}^{d-s-i})=e_0^{d-s}e_{d-s}^{d-s}
\end{equation}
for $0\le i\le d-s$, and equality holds in this equation for all $i$ if and only if equality holds in 2). Taking $(d-s)$-th roots, we have that equality holds in (\ref{eqM1}) if and only if equality holds in 3) for $0\le i\le d-s$.

\end{proof}

Teissier \cite{T3} (for Cohen-Macaulay normal two-dimensional complex analytic $R$), Rees and Sharp \cite{RS} (in dimension 2) and Katz \cite{Ka} (in complete generality)
have shown that 
 if $R$ is a formally equidimensional  local ring and $I(1), I(2)$  are $m_R$-primary ideals then the Minkowski inequality is an equality, that is, 
\begin{equation}
e_R(I(1)I(2))^{\frac{1}{d}}=e(I(1))^{\frac{1}{d}}+e(I(2))^{\frac{1}{d}},
\end{equation}
if and only if 
there exist positive integers $a$ and $b$ such that 
\begin{equation}\label{Mink2}
\overline{\sum_{n\ge 0}I(1)^{an}t^{n}}=\overline{\sum_{n\ge 0}I(2)^{bn}t^{n}}.
\end{equation}

 If $\mathcal I(1)$ and $\mathcal I(2)$ are filtrations of an analytically unramified  local ring $R$ 
 and condition \ref{Mink2}) holds then the Minkowski equality  holds between $\mathcal I(1)$ and $\mathcal I(2)$ by Theorem \ref{TheoremN10}, but the converse statement, that  the Minkowski equality  implies condition (\ref{Mink2})   is not true for filtrations, even for $m_R$-filtrations in a regular local ring, as is shown in a simple example in \cite{CSS}.

In \cite[Theorem 1.6]{C4}, it is shown that this characterization holds if $\mathcal I(1)$ and $\mathcal I(2)$ are bounded  filtrations of $m_R$-primary ideals in a $d$-dimensional analytically irreducible or excellent local domain ($s$-divisorial filtrations are bounded; bounded filtrations are defined in the following Section \ref{SecBound}). We show here that this characterization holds for  $s$-divisorial filtrations.

\begin{Theorem}\label{Theorem10} Suppose that $R$ is a $d$-dimensional  excellent local domain and that $\mathfrak a$ is an $m_R$-primary ideal. 
Let $\mathcal I(D_1)$ and $\mathcal I(D_2)$ be  two nontrivial integral  $s$-divisorial filtrations. Then  the Minkowski equality 
$$
e_s(\mathfrak a,\mathcal I(D_1)\mathcal I(D_2))^{\frac{1}{d-s}}= e_0^{\frac{1}{d-s}}+e_{d-s}^{\frac{1}{d-s}}
$$
holds between $\mathcal I(D_1)$ and $\mathcal I(D_2)$ if and only if   there exist $a,b\in \ZZ_{>0}$ such that $I(amD_1)=I(bmD_2)$ for all $m\in \NN$.
\end{Theorem}

\begin{proof} Let $\phi:X\rightarrow \mbox{Spec}(R)$ be a representation of the filtrations $\mathcal I(D_1)$ and $\mathcal I(D_2)$; that is, there are prime ideals $\mathfrak p_1,\ldots, \mathfrak p_t$ in $R$ such that $\dim R/\mathfrak p_i=s$ for all $i$ and $X$ is the  normalization of the blowup of an ideal $K$ of $R$ such that $\mathfrak p_1,\ldots, \mathfrak p_t$ are the associated primes of $K$. Further,  $D_1=\sum_{i=1}^tD_1(i)$ and $D_2=\sum_{i=1}^tD_2(i)$ are effective Weil divisors on $X$ such that for all $i$ and for $j=1,2$,  all prime components $E$ of $D_j(i)$ satisfy $m_E\cap R=\mathfrak p_i$, where $m_E$ is the maximal ideal of $\mathcal O_{X,E}$. Let 
$\mathcal I(D_j(i))=\{I(mD_j(i))\}$ be the associated divisorial $s$-filtrations on $R$. 



Let 
$$
P(n_1,n_2)=\lim_{m\rightarrow\infty}\frac{e_s(\mathfrak a,R/I(mn_1D_1)I(mn_2D_2))}{m^{d-s}/(d-s)!}
$$
and for $1\le i\le t$, let
$$
P_i(n_1,n_2)=\lim_{m \rightarrow \infty}\frac{e_s(\mathfrak a,R/I(mn_1D_1(i)I(mn_2D_2(i)))}{m^{d-s}/(d-s)!}.
$$
Write
$$
P(n_1,n_2)=\sum_{j=0}^{d-s}\frac{(d-s)!}{(d-s-j)!j!}e_jn_1^{d-s-j}n_2^j
$$
and
$$
P_i(n_1,n_2)=\sum_{j=0}^{d-s}\frac{(d-s)!}{(d-s-j)!j!}e(i)_jn_1^{d-s-j}n_2^j.
$$
We have that 
\begin{equation}\label{eqE1}
P(n_1,n_2)=\sum_{i=1}^t P_i(n_1,n_2)
\end{equation}
 by Theorem \ref{Theorem5} applied to $P(n_1,n_2)$ and the $P_i(n_2,n_2)$. 
 
First assume that  the Minkowski equality holds.
By our assumptions, the  conclusions of Lemma \ref{LemnMineq} hold for $\mathcal I(D(1))$ and $\mathcal I(D(2))$. 
There exists $i$ such that $D_1(i)\ne 0$. We have
\begin{equation}\label{eqF1}
e_{R_{\mathfrak p_i}}(\mathcal I(D_1(i))_{\mathfrak p_i}^{[d-s]},\mathcal I(D_2(i))_{\mathfrak p_i}^{[0]})
=e_{R_{\mathfrak p_i}}(\mathcal I(D_1(i))_{\mathfrak p_i})
\end{equation}
by \cite[Proposition 6.5]{CSS}.

Let $S$ be the normalization of $R$, which is dominated by $X$, and let $\mathfrak q_1,\ldots,\mathfrak q_u$ be the prime ideals of $S$ which lie over $\mathfrak p_i$, so that each $S_{\mathfrak q_i}$ is analytically irreducible. Then by Equation (18), Lemma 5.2 and Proposition 5.3 of \cite{C3}, we have that
\begin{equation}\label{eqF3}
e_{R_{\mathfrak p_i}}(\mathcal I(D_1(i))_{\mathfrak p_i})>0.
\end{equation}
Now 
\begin{equation}\label{eqF2}
e(i)_0=e_s(\mathfrak a,\mathcal I(D_1(i))^{[d-s]},\mathcal I(D_2(i))^{[0]})
=e_{R_{\mathfrak p_i}}(\mathcal I(D_1(i))_{\mathfrak p_i}^{[d-s]},\mathcal I(D_2(i))_{\mathfrak p_i}^{[0]})e_s(\mathfrak a,R/\mathfrak p_i)
\end{equation}
by Theorem \ref{Theorem5}. Thus $e(i)_0>0$ by (\ref{eqF3}) and so 
$$
e_0=\sum e(i)_0>0.
$$
Similarly there exists $i'$ such that $D_2(i')\neq 0.$ Then using   a similar argument to the above, we have  $e(i')_{d-s}>0$ and hence $e_{d-s}>0$.  Now $e_j^{d-s}=e_0^{d-s-j}e_{d-s}^j$ for $0\le j\le d-s$ by i) of Lemma \ref{LemnMineq}, since the Minkowski equality holds. Since $e_0>0$ and $e_{d-s}>0$, we have that $e_j>0$ for all $j$.

 Since the Minkowski equality holds between $\mathcal I(D_1)$ and $\mathcal I(D_2)$,  and $e_j>0$ for all $j$,  we have 
 by iii) of Lemma \ref{LemnMineq}
 that  the equalities 1) of Lemma \ref{LemnMineq} hold so  there exists $\xi\in \RR_{>0}$ such that
\begin{equation}\label{eqT11}
\xi=\frac{e_1}{e_0}=\cdots=\frac{e_{d-s}}{e_{d-s-1}}.
\end{equation}

By (\ref{eqE1}) we have that $e_j=\sum_{i=1}^t e(i)_j$ for all $j$. By the inequalities 
$e(i)_j^2\le e(i)_{j-1}e(i)_{j+1}$ for $1\le j\le d-s-1$ (Theorem \ref{Mink1}) and (\ref{eqT11}) we have that
$$
\begin{array}{lll}
0&\le& \sum_{i=1}^t(e(i)_{j+1}^{\frac{1}{2}}-\xi e(i)_{j-1}^{\frac{1}{2}})^2
=\sum_{i=1}^t(e(i)_{j+1}-2\xi e(i)_{j+1}^{\frac{1}{2}}e(i)_{j-1}^{\frac{1}{2}}+\xi^2e(i)_{j-1})\\
&\le& \sum_{i=1}^t(e(i)_{j+1}-2\xi e(i)_j+\xi^2 e(i)_{j-1})\\
&=& e_{j+1}-2\xi e_j+\xi^2e_{j-1}\\
&=&\xi^2 e_{j-1}-2\xi^2 e_{j-1}+\xi^2 e_{j-1}=0.
\end{array}
$$
 Thus 
\begin{equation}\label{eqN1}
e(i)_{j+1}^{\frac{1}{2}}=\xi e(i)_{j-1}^{\frac{1}{2}}\mbox{ and }e(i)_j^2=e(i)_{j-1}e(i)_{j+1}
\end{equation}

for all $i$ and $1\le j\le d-s-1$  respectively. 
Thus for a particular $i$, either 
\begin{equation}\label{eqG1}
e(i)_j=0\mbox{ for all }j
\end{equation}
or
\begin{equation}\label{eqG2}
e(i)_j>0\mbox{ for all }j.
\end{equation}

Suppose that (\ref{eqG1}) holds for a particular $i$. Then $e(i)_0=e(i)_{d-s}=0$ which implies that
\begin{equation}\label{eqF4}
e_{R_{\mathfrak p_i}}(\mathcal I(D_1(i))_{\mathfrak p_i})=e_{R_{\mathfrak p_i}}(\mathcal I(D_2(i))_{\mathfrak p_i})=0
\end{equation}
and so $\mathcal I(D_1(i))_{\mathfrak p_i}$ is  the trivial filtration since we have a contradiction  to (\ref{eqF4}) by (\ref{eqF1}) and (\ref{eqF3})  if  $\mathcal I(D_1(i))_{\mathfrak p_i}$ is not trivial.  Replacing $D_1(i)$ with $D_2(i)$ in this argument we obtain that 
$\mathcal I(D_2(i))_{\mathfrak p_i}$ is also the trivial filtration. 
In particular, $I(D_1(i))R_{\mathfrak p_i}=R_{\mathfrak p_i}$ and $I(D_2(i))R_{\mathfrak p_i}=R_{\mathfrak p_i}$, a contradiction, since $\mathfrak p_i$ must be an associated prime of at least one of $I(D_1(i))$ or $I(D_2(i))$. Thus (\ref{eqG1}) cannot hold, and so (\ref{eqG2}) holds for all $i$.



Let us now consider a particular $i$. Then by
(\ref{eqN1}) and Lemma \ref{LemnMineq},  the Minkowski equalities  of Lemma \ref{LemnMineq} hold between $\mathcal I(D(i)_1)$ and $\mathcal I(D(i)_2)$. Thus there exists $\lambda_i\in \RR_{>0}$ such that 
$$
\frac{e(i)_{j+1}}{e(i)_j}=\lambda_i
$$
 for all $j$. Thus for $1\le j\le d-s-1$,
 $$
 \xi^2=\frac{e(i)_{j+1}}{e(i)_{j-1}}=\frac{e(i)_{j+1}}{e(i)_j}\frac{e(i)_j}{e(i)_{j-1}}
 =\lambda_i^2
 $$
 so that $\lambda_i=\xi$  and so
 $$
 \frac{e(i)_{d-s}^{\frac{1}{d-s}}}{e(i)_0^{\frac{1}{d-s}}}=\xi.
 $$
 We have that
 $$
 e(i)_j=e_s(\mathfrak a,\mathcal I(D_1(i))^{[d-s-j]},\mathcal I(D_2(i))^{[j]})
 =e_{R_{\mathfrak p_i}}(\mathcal I(D_1)_{\mathfrak p_i}^{[d-s-j]},\mathcal I(D_2)_{\mathfrak p_i}^{[j]})e_s(\mathfrak a,R/\mathfrak p_i)
 $$
 by Theorem \ref{Theorem5}. Thus for each $i$, the
 $e_{R_{\mathfrak p_i}}(\mathcal I(D_1)_{\mathfrak p_i}^{[d-s-j]},\mathcal I(D_2)_{\mathfrak p_i}^{[j]})$ satisfy the Minkowski equalities 1) - 3) of Lemma  \ref{LemnMineq} with 
 $$
 \frac{e_{R_{\mathfrak p_i}}(\mathcal I(D_2)_{\mathfrak p_i})^{\frac{1}{d-s}}}
 {e_{R_{\mathfrak p_i}}(\mathcal I(D_1)_{\mathfrak p_i})^{\frac{1}{d-s}}} 
 =\xi.
 $$
 Thus $\xi\in \QQ$ by \cite[Theorem 12.1]{C3} and the  proof of this theorem.

 Write $\xi=\frac{a}{b}$ with $a,b\in \ZZ_{>0}$. We have that $I(maD_1(i)_{\mathfrak p_i})=I(mbD_2(i)_{\mathfrak p_i})$ for all $i$  and $m\in \NN$ by \cite[Theorem 12.1]{C3}. Since  the only associated prime of $I(maD_1(i))$ and $I(mbD_2(i))$ is $\mathfrak p_i$, we have that $I(maD_1(i))=I(mbD_2(i))$ for all $i$  and $m\in \NN$. 
 Thus 
 $$
 I(maD_1)=I(maD_1(1))\cap \cdots\cap I(maD_1(t))=I(mbD_2(1))\cap \cdots\cap I(mbD_2(t))
 =I(mbD_2)
  $$
 for all $m\in \NN$.
 
  The converse follows from Theorem \ref{TheoremN10} since $\overline{R[\mathcal I(D_j)]}=R[\mathcal I(D_j)]$ for $j=1,2$.

  \end{proof}

The following corollary is proven for $m_R$-valuations (divisorial valuations which dominate $m_R$) in  \cite[Corollary 12.2]{C3}. 

\begin{Corollary}\label{Cor1.10} Suppose that $R$ is an excellent local domain and $\nu_1$ and $\nu_2$ are divisorial valuations of the quotient field of $R$  such that the valuation rings $\mathcal O_{\nu_1}$ and $\mathcal O_{\nu_2}$ both contain $R$. Suppose  that $s=\dim R/(m_{\nu_1}\cap R)=\dim R/(m_{\nu_2}\cap R)$ and  
Minkowski's equality holds between the filtrations $\mathcal I(\nu_1)=\{I(\nu_1)_m\}$ and $\mathcal I(\nu_2)=\{I(\nu_2)_m\}$. Then $\nu_1=\nu_2$.
\end{Corollary}

 \begin{proof}
We have by Theorem \ref{Theorem10} that $I(\nu_1)_{an}=I(\nu_2)_{bn}$ for all $n$ and some positive integers $a$ and $b$ which we can take to be relatively prime.

Suppose that $0\ne f\in I(\nu_1)_n$. Then $f^a\in I(\nu_1)_{an}=I(\nu_2)_{bn}$ so that $a\nu_2(f)\ge bn$. If $f^a\in I(\nu_2)_{bn+1}$ then $f^{ab}\in I(\nu_2)_{b(bn+1)}=I(\nu_1)_{a(bn+1)}$ so that $\nu_1(f)>n$. Thus
\begin{equation}\label{eqX21}
\nu_1(f)=n\mbox{ if and only if }\nu_2(f)=\frac{b}{a}n.
\end{equation}
Further, (\ref{eqX21}) holds for every nonzero $f\in {\rm QF}(R)$ since $f$ is a quotient of nonzero elements of $R$.

Now the maps 
 $\nu_1:{\rm QF}(R)\setminus \{0\}\rightarrow \ZZ$ and 
$\nu_2:{\rm QF}(R)\setminus \{0\}\rightarrow \ZZ$ are surjective, so there exists $0\ne f\in {\rm QF}(R)$ such that $\nu_1(f)=1$  and there exists $0\ne g\in {\rm QF}(R)$ such that $\nu_2(g)=1$ which implies that $a=b=1$ since $a,b$ are relatively prime. Thus $\nu_1=\nu_2$.

\end{proof}

  \section{Bounded filtrations} \label{SecBound}
  
  \begin{Definition} Suppose that $R$ is a 
   local domain. A filtration $\mathcal I=\{I_n\}$ on $R$ is said to be a bounded filtration if there exists an integral divisorial filtration $\mathcal I(D)$ on $R$ such that
$$
\overline{R[\mathcal I]}=R[\mathcal I(D)].
$$
A filtration $\mathcal I=\{I_n\}$ on $R$ is said to be a bounded $s$-filtration if there exists an integral $s$-divisorial filtration $\mathcal I(D)$ on $R$ such that
$$
\overline{R[\mathcal I]}=R[\mathcal I(D)].
$$

A filtration $\mathcal I=\{I_n\}$ on $R$ is said to be a real bounded $s$-filtration if there exists a real $s$-divisorial filtration $\mathcal I(D)$ on $R$ such that
$$
\overline{R[\mathcal I]}=R[\mathcal I(D)].
$$
\end{Definition}

\begin{Lemma}\label{BdLemma3} Suppose that $R$ is an excellent local domain and $\mathcal I=\{I^n\}$ 
is the filtration of powers of a fixed  ideal $I$. Then $\mathcal I$ is bounded.
\end{Lemma}

The proof of Lemma \ref{BdLemma3}  for $m_R$-filtrations in \cite[Lemma 5.9]{C3} extends immediately to arbitrary divisorial filtrations.   

\begin{Proposition}\label{BoundProp} Suppose that $R$ is a local ring with $\dim N(\hat R)<d$,  $\mathfrak a$ is an $m_R$-primary ideal  and 
$$
\mathcal I(1),\ldots,\mathcal I(r),\mathcal I'(1),\ldots,\mathcal I'(r)
$$
 are $s$-filtrations such that 
$\overline{R(\mathcal I'(i))}=\overline{R(\mathcal I(i))}$ for $1\le i\le r$. Then
we have equality of all mixed multiplicities
\begin{equation}\label{eqB1}
e_s(\mathfrak a,\mathcal I(1)^{[d_1]},\ldots,\mathcal I(r)^{[d_r]})
=
e_s(\mathfrak a,\mathcal I'(1)^{[d_1]},\ldots,\mathcal I'(r)^{[d_r]}).
\end{equation}
\end{Proposition}

The proof is as the proof for $m_R$-filtrations, given in \cite[Proposition 5.10]{C3}. The references to \cite[Theorem 6.9]{CSS} 
and \cite[Appendix]{C4} in the proof in \cite{C3} must be replaced with a reference to Theorem \ref{Rees1} of this paper.

\begin{Theorem}\label{RBT} Suppose that $R$ is an excellent local domain, $\mathfrak a$ is an $m_R$-primary ideal, $\mathcal I(1)$ is a real  bounded $s$-filtration and $\mathcal I(2)$ is an arbitrary filtration  such that $\mathcal I(1)\subset \mathcal I(2)$, so that $s(\mathcal I(2))\le s=s(\mathcal I(1))$.  Then    $e_s(\mathfrak a,\mathcal I(1))=e_s(\mathfrak a,\mathcal I(2))$ if and only if there is equality of  integral closures  
$$
\overline{\sum_{m\ge 0}I(1)_mt^m}=\overline{\sum_{m\ge 0}I(2)_mt^m}
$$
  in $R[t]$.
 \end{Theorem}
 
 \begin{proof}  First suppose that there is equality of integral closures
  $$
\overline{\sum_{m\ge 0}I(1)_mt^m}=\overline{\sum_{m\ge 0}I(2)_mt^m}
$$ 
in $R[t]$. Then $e_s(\mathfrak a,\mathcal I(1))=e_s(\mathfrak a,\mathcal I(2))$ by Theorem \ref{Rees1}. 

Now suppose that $e_s(\mathfrak a,\mathcal I(1))=e_s(\mathfrak a,\mathcal I(2))$.
 Let $\mathcal I(D_1)$ be  the real divisorial $s$-filtration such that $\overline{R[\mathcal I(1)]}=R[\mathcal I(D_1)]$. 
 Let $\mathcal J$ be the filtration on $R$ such that $R[\mathcal J]=\overline{R[\mathcal I(2)]}$.
 Then 
 $$
 R[\mathcal I(D_1)]\subset \overline{R[\mathcal I(2)]}=R[\mathcal J]
 $$
  so that $\mathcal I(D_1)\subset \mathcal J$. We have that $e_s(\mathcal I(1))=e_s(\mathcal I(D_1))$ and 
  $e_s(\mathcal I(2))=e_s(\mathcal J)$ by Theorem \ref{Rees1}. Thus $e_s(\mathcal I(D_1))=e_s(\mathcal J)$ and so $R[\mathcal I(D_1)]=R[\mathcal J]$ by Theorem \ref{Theorem9}.
 
 Thus the conclusion of the theorem holds for $\mathcal I(1)$ and $\mathcal I(2)$.
   \end{proof}

\begin{Example}\label{Example1} Theorem \ref{RBT} does not extend to arbitrary bounded filtrations, or to arbitrary divisorial filtrations. 
\end{Example}
There exist bounded filtrations $\mathcal I(1)\subset \mathcal I(2)$ with $s(\mathcal I(1))=s(\mathcal I(2))$ 
and $e_s(m_R,\mathcal I(1))=e_s(m_R,\mathcal I(2))$ but $\overline{R[\mathcal I(1)]}\ne \overline{R[\mathcal I(2)]}$.
Let $k$ be an algebraically closed field and let $R=k[x,y,z]_{(x,y,z)}$, a local ring of the polynomial ring over $k$ in three variables. 
Let $\mfp$ be a height two prime ideal in $R$ such that  the symbolic algebra $\oplus_{n\ge 0}\mathfrak p^{(n)}$ is not a finitely generated $R$-algebra. Some examples where this algebra is not finitely generated are given  in \cite{Ro} and \cite{GNW}. The $\mathfrak p$-adic filtration $\mathcal I(1)=\{\mathfrak p^n\}$ is bounded by Lemma \ref{BdLemma3} and the filtration of symbolic powers of $\mathfrak p$, $\mathcal I(2)=\{\mathfrak p^{(n)}\}$, is a divisorial filtration. We have that $s(\mathcal I(1))=s(\mathcal I(2))=1$ and since  $\mathfrak p^nR_{\mathfrak p}=\mathfrak p^{(n)}R_{\mathfrak p}$ for all $n$, we have by Proposition \ref{Prop2} that 
$$
e_1(m_R,\mathcal I(1))=e_{R_{\mathfrak p}}(\mathcal I(1)_{\mathfrak p})e_{m_R}(R/\mathfrak p)
=e_{R_{\mathfrak p}}(\mathcal I(2)_{\mathfrak p})e_{m_R}(R/\mathfrak p)
=e_1(m_R,\mathcal I(2)).
$$
But we cannot have that $R[\mathcal I(2)]$ is integral over $R[\mathcal I(1)]$ since its integral closure $\overline{R[\mathcal I(1)]}$
 is a finitely generated $R$-algebra, and $R[\mathcal I(2)]=\sum_{n\ge 0}\mathfrak p^{(n)}t^n$ is not.  The reason for this example is because of the existence of embedded primes in the filtration $\mathcal I(1)$.
 
 We now modify the example to show that Theorem \ref{RBT} does not extend to divisorial filtrations. Let $\mathcal I(3)$ be the filtration $\mathcal I(3)=\{\overline{\mathfrak p^n}\}$. Let $X$ be the  normalization of the blowup of $\mathfrak p$. Then 
 $\mathfrak p\mathcal O_X$ is invertible on $X$ and
  $\overline{\mathfrak p^n}=\Gamma(X,\mathfrak p^n\mathcal O_X)\cap R$ for all $n$, so that $\mathcal I(3)$ is a divisorial filtration on $R$.  
  Since $R[\mathcal I(3)]=\overline{R[\mathcal I(1)]}$ by Remark \ref{Rem1}, we have that $e_1(m_R,\mathcal I(3))=e_1(m_R,\mathcal I(1))$ by Theorem \ref{Rees1}. Thus $e_1(m_R,\mathcal I(3))=e_1(m_{R_1},\mathcal I(2))$. Now $R[\mathcal I(3)]\subset R[\mathcal I(2)]$ since $R[\mathcal I(2)]$ is integrally closed in $R[t]$ by Lemma \ref{LemmaICD}. Thus $\mathcal I(3)\subset \mathcal I(2)$.
  
  \begin{Theorem}\label{TRSKT} Suppose that $R$ is a $d$-dimensional excellent local domain, $\mathfrak a$ is an $m_R$-primary ideal   and $\mathcal I(1)$ and $\mathcal I(2)$ are  two nontrivial bounded $s$-filtrations. Then the following are equivalent. 
\begin{enumerate}
\item[1)]  The Minkowski inequality
$$
e_s(\mathfrak a,\mathcal I(1)\mathcal I(2))^{\frac{1}{d-s}}=e_s(\mathfrak a,\mathcal I(1))^{\frac{1}{d-s}}+e_s(\mathfrak a,\mathcal I(2))^{\frac{1}{d-s}}
$$
holds.
\item[2)] There exist positive integers $a,b$ such that there is equality of  integral closures
$$
\overline{\sum_{n \ge 0}I(1)_{an}t^n}=\overline{\sum_{n\ge 0}I(2)_{bn}t^n}
$$
 in $R[t]$.
\end{enumerate}
 \end{Theorem}
 
 \begin{proof} Let $\mathcal I(D_1)$ and $\mathcal I(D_2)$ be  integral divisorial $s$-filtrations such that $\overline{R(\mathcal I(1))}=R(\mathcal I(D_1))$ and $\overline{R(\mathcal I(2))}=R(\mathcal I(D_2))$. By Proposition \ref{BoundProp}, we have equality of functions 
 $$
 \lim_{m\rightarrow \infty} \frac{e_s(\mathfrak a,R/I(i)_{mn_1}I(i)_{mn_2})}{m^d/(d-s)!}
 =\lim_{m\rightarrow \infty} \frac{e_s(\mathfrak a,R/I(D_i)_{mn_1}I(D_i)_{mn_2})}{m^d/(d-s)!}
  $$
  for $i=1,2$ and all $n_1,n_2\in \NN$. Since 1) and 2) are equivalent for the integral $s$-divisorial filtrations $\mathcal I(D_1)$ and $\mathcal I(D_2)$ by Theorem \ref{Theorem10}, they are also equivalent for the bounded $s$-filtrations  $\mathcal I(1)$ and $\mathcal I(2)$.
 \end{proof}
 
 \begin{Example}\label{Example2} Theorem \ref{TRSKT} does not extend to arbitrary bounded  filtrations or to arbitrary divisorial filtrations. 
\end{Example}

 The example constructed in Example \ref{Example1} gives an example. We have that $d=\dim R=3$, $s=\dim R/p=1$,  $d-s=\dim R_p=2$ and $\mathfrak p^nR_{\mathfrak p}=\mathfrak p^{(n)}R_{\mathfrak p}$ for all $n$, so that $\mathcal I(1)_{\mathfrak p}=\mathcal I(2)_{\mathfrak p}$ are $0$-divisorial filtrations on $R_{\mathfrak p}$. Thus 
 $$
 e_0(m_{R_{\mathfrak p}},\mathcal I(1)_{\mathfrak p}\mathcal I(2)_{\mathfrak p})^{\frac{1}{\dim R_{\mathfrak p}}}
 =e_0(m_{R_{\mathfrak p}},\mathcal I(1)_{\mathfrak p})^{\frac{1}{\dim R_{\mathfrak p}}}+e_0(m_{R_{\mathfrak p}},\mathcal I(2)_{\mathfrak p})^{\frac{1}{\dim R_{\mathfrak p}}}
   $$
   by Theorem \ref{TRSKT}. Since 
   $$
   \begin{array}{l}
   e_s(m_R,\mathcal I(1))=e_0(m_{R_{\mathfrak p}},\mathcal I(1)_{\mathfrak p})
   e_{m_R}(R/\mathfrak p),\,\,
   e_s(m_R,\mathcal I(2))=e_0(m_{R_{\mathfrak p}},\mathcal I(2)_{\mathfrak p})
   e_{m_R}(R/\mathfrak p)\mbox{ and }\\   e_s(m_R,\mathcal I(1)\mathcal I(2))=e_0(m_{\mathfrak p},\mathcal I(1)_{\mathfrak p}\mathcal I(2)_{\mathfrak p})
   e_{m_R}(R/\mathfrak p)
   \end{array}
   $$ 
 we have that the Minkowski equality
  $$
 e_s(m_R,\mathcal I(1)\mathcal I(2))^{\frac{1}{d-s}}
 =e_s(m_R,\mathcal I(1))^{\frac{1}{d-s}}+e_s(m_R,\mathcal I(2))^{\frac{1}{d-s}}
   $$
 holds between $\mathcal I(1)$ and $\mathcal I(2)$, but as shown in  Example \ref{Example1}, there do not exist $a,b\in \ZZ_{>0}$ such that $\overline{\mathfrak p^{an}}=\mathfrak p^{(bn)}$ for all $n\in \NN$.
 
 Similarly, we have that $\mathcal I(3)\subset \mathcal I(2)$ are divisorial filtrations which satisfy the Minkowski equality, but there do not exist $a,b\in \ZZ_{>0}$ such that $\overline{\mathfrak p^{an}}=\mathfrak p^{(bn)}$ for all $n\in \NN$.

\section{equimultiple ideals and bounded $s$-filtrations}
  
 The analytic spread of an ideal $I$ in a local ring $R$ is defined to be
  $$
  \ell(I)=\dim R[It]/m_RR[It].
  $$
  We have  inequalities 
  $$
  {\rm ht}(I)\le \ell(I)\le \dim R,
  $$
 proven for instance in \cite[Corollary 8.3.9]{HS}. An ideal $I$ for which the equality $\mbox{ht}(I)=\ell(I)$ holds is called equimultiple. 
 
  B\"oger  generalized Rees's theorem to equimultiple ideals in a formally equidimensional local ring.  
  
 \begin{Theorem}(\cite{EB}, also \cite[Corollary 11.3.2]{HS}) Suppose that  $R$ is a formally equidimensional local ring and $I\subset J$ are two ideals such that $\ell(I)=\mbox{ht}(I)$ ($I$ is equimultiple). Then $J\subset \overline{I}$ if and only $e_{R_P}(I_P)=e_{R_P}(J_P)$ for every prime $P$ minimal over $I$.
 \end{Theorem}
 
 Let $I$ be an ideal in an excellent local domain $R$ and  $R[I]=\bigoplus_{n\ge 0}I^n$ be the Rees algebra of $R$,
 $\pi: X=\mbox{Proj}(R[I])\rightarrow \mbox{Spec}(R)$ be the blowup of $I$. Let $B$ be the normalization of $R[I]$ in its quotient field, which is a finitely generated graded $R$-algebra. Let $Y=\mbox{Proj}(B)$, the ``normalized blowup'' of $I$. Let $\alpha:Y\rightarrow \mbox{Spec}(R)$ be the  natural composition map $Y\rightarrow X\stackrel{\pi}{\rightarrow}\mbox{Spec}(R)$.
 
 Let $p\in \mbox{Spec}(R)$ and  $\kappa(p):=(R/p)_p$. Then (by definition) 
 $$
 \pi^{-1}(p)= X\times_{\mbox{Spec}(R)}\mbox{Spec}(\kappa(p)) =\mbox{Proj}(R[I]\otimes_R \kappa(p))
 $$
  and 
  $$
  \alpha^{-1}(p)=Y\times_{\mbox{Spec}(R)}\mbox{Spec}(\kappa(p))=\mbox{Proj}(B\otimes_R\kappa(p)).
 $$
 
 Since $Y\rightarrow X$ is finite, we have (by \cite[Theorems A.6 and A.7]{BH}) that if $p\in \mbox{Spec}(R)$, then 
$$
\dim \pi^{-1}(p)=\dim\alpha^{-1}(p).
$$
We also have that ``upper semicontinuity of fiber dimension'' holds; that is, if $p\subset p'$ are prime ideals in $R$, then  
$$
\dim \pi^{-1}(p)\le \dim \pi^{-1}(p')
$$
by \cite[(IV.13.1.5)]{EGA}.
 
 Write $I\mathcal O_{Y}=\mathcal O_Y(-a_1E_1-\cdots -a_rE_r)$, where $E_1,\ldots,E_r$ are prime Weil divisors on $Y$. Then $I\mathcal O_Y$ is an ample Cartier divisor on $Y$. Let $\nu_1,\ldots,\nu_r$ be the corresponding valuations on the quotient field of $R$. We have that for all $n\ge 0$,
 $$
 \overline{I^n}=\Gamma(Y,I^n\mathcal O_Y)\cap R= I(\nu_1)_{a_1n}\cap \cdots\cap I(\nu_r)_{a_rn}
 $$
 is a  primary decomposition of $\overline{ I^n}$.

 Now we have that for $p\in \mbox{Spec}(R)$,
 $\dim \alpha^{-1}(p)\le\dim R_p-1$  and $\dim \alpha^{-1}(p)=\dim R_p-1$ if and only if some $\nu_i$ dominates $p$. Further,
 if $p$ does not contain $I$, then $\dim \alpha^{-1}(p)=0$, and if $p$ is a minimal prime of $I$, then $\dim \alpha^{-1}(p)=\dim R_p-1$.

 \begin{Proposition}\label{PropAS30*} Suppose that $I$ is an ideal of an excellent local domain $R$ and $I$ is equimultiple. Let $s=\dim R-\mbox{ht}(I)$. Then there exist divisorial valuations $\nu_1,\ldots,\nu_r$ such that the center of $\nu_i$ on $R$ has dimension $s$ for all $i$, and $a_1,\ldots, a_r\in \ZZ_{>0}$ such that 
 $$
 \overline {I^n}=I(\nu_1)_{a_1n}\cap\cdots\cap I(\nu_r)_{a_rn}
 $$
 for all $n\in \NN$.
 \end{Proposition}
 
 \begin{proof} By our assumption, $\dim \alpha^{-1}(m_R)=\mbox{ht}(I)-1$. Suppose that $p\in \mbox{Spec}(R)$.   If $p$ is  dominated by some $\nu_i$, then $I\subset p$ and
 $$
 \mbox{ht}(I)-1\le \dim R_p-1=\dim \alpha^{-1}(p)\le \dim \alpha^{-1}(m_R)=\mbox{ht}(I)-1,
 $$
 so that $\mbox{ht}(p)=\mbox{ht}(I)$. 
\end{proof}

\begin{Corollary}\label{CorAS31*} Suppose that $R$ is an excellent local domain and $I$ is an equimultiple ideal on $R$. Then the $I$-adic filtration $\mathcal I=\{I^n\}$ is a bounded $s$-filtration, where $s=\dim R-\mbox{ht}(I)$.
\end{Corollary}
 
 A much more difficult to prove  form of Proposition \ref{PropAS30*} is true in a locally formally equidimensional Noetherian ring $A$. If $I$ is an equimultiple ideals in $A$, then for all $n\ge 0$, every associated prime of $\overline{I^n}$ has height $\ell=\ell(I)$. This statement  follows from \cite[Theorem 2.12]{Ra}. By \cite[Lemmas 8.42 and B.47]{HS}, we may assume that $R$ has an infinite residue field. By \cite[Proposition 8.3.7]{HS}, $I$ then has a reduction $J$ generated by $\ell$ elements. By \cite[Theorem 2.12]{Ra} or \cite[Corollary 5.4.2]{HS}, every associated prime of $\overline{J^n}=\overline{I^n}$ has height $\ell$.
 
 \begin{Example}\label{ExAS1} There exists a height one prime ideal $P$ in a normal, excellent 3 dimensional local ring $R$ and $d>0$  such that 
    the $P$-adic filtration  $\{P^n\}_{n\in \NN}$ of $R$ is a bounded 2-filtration but 
  $2=\ell(P)>\mbox{ht}(P)=1$, so that $P$ is not equimultiple.
    \end{Example}

  \begin{proof} Let $k$ be a field and   $k[x,y,z,w]$ be a polynomial ring over $k$. Let  
  $$R
  =(k[x,y,z,w]/(xy-zw))_{(x,y,z,w)}.
  $$
   Let $\overline x, \overline y, \overline z, \overline w$ be the respective classes of $x,y,z,w$ in $R$.
  Let $P=(\overline x,\overline z)$, which is a height 1 prime ideal in $R$.
  
 The blowup   $\pi:X=\mbox{Proj}(\oplus_{n\ge 0}P^n)\rightarrow \mbox{Spec}(R)$ of $P$ is such that $X$ is nonsingular, and $\pi^{-1}(m_R)\cong \PP_k^1$. This simple and well known calculation is outlined in Exercise 6.16 on page 125 of \cite{AG}.
  Thus
 $\ell(P)=\dim \pi^{-1}(m_R)+1=2$, and so $1=\mbox{ht}(P)<\ell(P)=2$ and so $P$ is not equimultiple.
  
 Let $E=\mbox{Spec}(\mathcal O_X/P\mathcal O_X)$, an integral surface on $X$, so that  $\mathcal O_{X,E}$ is a valuation ring, with associated valuation $\nu_E$. Now $P^n\mathcal O_X=\mathcal O_X(-nE)$ for all $n$, and
 $$
 \overline{P^n}=\pi_*(P^n\mathcal O_X)=\pi_*(\mathcal O_X(-nE))=P^{(n)}=I(\nu_E)_n
 $$
 for all $n\in \NN$. Thus $\{P^n\}$ is a bounded 2-filtration.
 \end{proof}
 
 We conclude from Corollary \ref{CorAS31*} and Example \ref{ExAS1}, that in an excellent local domain of dimension $d$, the set of $I$-adic filtrations of equimultiple height $r$ ideals is a strict subset of the set of bounded $(d-r)$-filtrations.  
 
 The following theorem is a  generalization in excellent local domains of B\"oger's theorem from equimultiple ideals to the larger class of $s$-filtrations. Theorem \ref{BoGen} is a consequence of Theorem \ref{RBT}.

\begin{Theorem}\label{BoGen}
	Suppose that $R$ is an excellent local domain, $\mathfrak a$ is an $m_R$-primary ideal, $\mathcal I(1)$ is a real  bounded $s$-filtration and $\mathcal I(2)$ is an arbitrary filtration with $\mathcal I(1)\subset \mathcal I(2)$.  Then  $e_{R_{\mfp}}(\mathcal I(1)_\mfp)=e_{R_{\mfp}}(\mathcal I(2)_\mfp)$ for all $\mfp\in \Min(R/I(1)_1)$ if and only if there is equality of  integral closures  
	\begin{equation}\label{eqBoGen}
	\overline{\sum_{m\ge 0}I(1)_mt^m}=\overline{\sum_{m\ge 0}I(2)_mt^m}
	\end{equation}
	in $R[t]$.	
	\end{Theorem}
	


\begin{proof} 
By Proposition \ref{Prop2}, we have that
\begin{equation}\label{eqBo1}
e_s(\mfa,\mathcal I(1))=\sum_{\mfp\in \Min(R/I(1)_1)}e_{R_{\mfp}}(\mathcal I(1)_{\mfp})e_{\mfa}(R/\mfp).
\end{equation}

Since $I(1)_1\subset I(2)_1$, we have that $V(I(2)_1)\subset V(I(1)_1)$. Thus $s(\mathcal I(2))\le s(\mathcal I(1))=s$. Since $V(I(1)_1)$  is equidimensional of dimension $s$, we have by Proposition \ref{Prop2} that 
\begin{equation}\label{eqBo2}
e_s(\mfa,\mathcal I(2))=\sum_{\mfp\in \Min(R/I(1)_1)}e_{R_{\mfp}}(\mathcal I(2)_{\mfp})e_{\mfa}(R/\mfp).
\end{equation}
Now $\mathcal I(1)\subset\mathcal I(2)$ implies
\begin{equation}\label{eqBo3}
e_{R_{\mfp}}(\mathcal I(1)_{\mfp})\ge e_{R_{\mfp}}(\mathcal I(2)_{\mfp})
\end{equation}
for all $\mfp\in \Min(R/I(1)_1)$. By equations (\ref{eqBo1}), (\ref{eqBo2}) and (\ref{eqBo3}), we have that
$e_s(\mfa,\mathcal I(1))=e_s(\mfa,\mathcal I(2))$ if and only if
$e_{R_{\mfp}}(\mathcal I(1)_{\mfp})=e_{R_{\mfp}}(\mathcal I(2))$ for all $\mfp\in \Min(R/I(1)_1)$. The theorem now follows from  Theorem \ref{RBT}, which shows that $e_s(\mfa,\mathcal I(1))=e_s(\mfa,\mathcal I(2))$ if and only if the equality of integral closures of (\ref{eqBoGen}) holds.
\end{proof}


\begin{thebibliography}{1000000000}
\bibitem{Bh} P.B. Bhattacharya, The Hilbert function of two ideals, Proc. Camb. Phil. Soc. 53 (1957), 568 - 575.
\bibitem{Bo} N. Bourbaki,  Commutative Algebra, Chapters 1-7, Springer Verlag, 1989.

\bibitem{EB} E. B\"oger, Einige Bemerkungen zur Theorie der ganzalgebraischen Abhandigkeit von Idealen, Math. Ann 185 (1970), 303-308.

\bibitem{Br} M. Brodmann, Asymptotic stability of $\mbox{Ass}(M/I^nM)$ Proc. Amer. Math. Soc. 74 (1979), 16-18.

\bibitem{BH} W. Bruns and J. Herzog, Cohen-Macaulay rings, Cambridge studies in Advanced Mathematics 39, Cambridge University Press, 1993, 13H10 (13-02).


\bibitem{CM} Y. Cid-Ruiz and J. Monta\~no, Mixed Multiplicities of Graded Families of Ideals, arXiv:2010.11862v1.


\bibitem{AG} S.D. Cutkosky, Introduction to Algebraic Geometry, Graduate studies in Mathematics 188, American Mathematical Society, 2018.
\bibitem{C1} S.D. Cutkosky, Asymptotic multiplicities of graded families of ideals and linear series,  Advances in Mathematics 264 (2014), 55 - 113.
\bibitem{C2} S.D. Cutkosky, Asymptotic Multiplicities, Journal of Algebra 442 (2015), 260 - 298.
\bibitem{C4} S.D. Cutkosky, Mixed Multiplicities of Divisorial Filtrations, Advances in Math. 358 (2019).
\bibitem{C3} S.D. Cutkosky, The Minkowski Equality of Filtrations, arXiv: 2007.06025v3.
\bibitem{C5} S.D. Cutkosky, Examples of multiplicities and mixed multiplicities of filtrations, arXiv:2007.03459.
\bibitem{CSS} S.D. Cutkosky, P. Sarkar and H. Srinivasan, Mixed Multiplicities of Filtrations, Trans. Amer. Math. Soc., 372 (2019), 6183 - 6211.

\bibitem{ELS} L. Ein, R. Lazarsfeld and K. Smith, Uniform Approximation of Abhyankar valuation ideals in smooth function fields, Amer. J. Math. 125 (2003), 409 - 440.

\bibitem{KG} K. Goel,  Mixed Multiplicities of Ideals, Lecture Notes, IIT Bombay.


\bibitem{GGV} K. Goel, R.V. Gurjar and J.K. Verma, The Minkowski's equality and inequality for multiplicity of ideals of finite length in Noetherian local rings,  Contemporary Math 738 (2019).

\bibitem{GNW} S. Goto, K. Nishida and  K. Watanabe, Non Cohen-Macaulay symbolic blow-ups for space monomial curves and counterexamples to Cowsik's question, Proc. Amer. Math. Soc. 120 (1994), 383- 392.





\bibitem{EGA} A. Grothendieck and J. Dieudonn\'e, \'El\'ements de G\'eom\'etrie Alg\'ebrique IV, Parts 2 and 3,  Publ. Math. IHES 24 (1965) and 28 (1966).

\bibitem{Ho} G. Hardy, J.E. Littlewood and G. P\'olya, Inequalities, second edition, Cambridge University Press, 1952.

\bibitem{Ha} R. Hartshorne, Algebraic Geometry, Graduate Texts in Mathematics, No. 52. Springer-Verlag, New York-Heidelberg, 1977.


	
	\bibitem{HIO} Herrmann, M., Ikeda, S., and Orbanz, U., Equimultiplicity and blowing up: an algebraic study, Springer-Verlag, Berlin, 1988.
	
	\bibitem{HSV} M. Herrman, R. Schmidt and W. Vogel, Theorie der normalen flaxhheit, B.G. Teubner, Leipzig, 1977.
	
	\bibitem{H} J. Huh, Milnor numbers of projective numbers of projective hypersurfaces and the chromatic polynomials of graphs, Journal of the Amer. Math. Soc. 25 (2012), 907 - 927.
	
\bibitem{Ka} D. Katz, Note on multiplicity, Proc. Amer. Math. Soc. 104 (1988), 1021 - 1026.


	
	
\bibitem{KV} D. Katz and J. Verma, Extended Rees algebras and mixed multiplicities, Math. Z. 202 (1989), 111-128.

\bibitem{KK} K. Kaveh and G. Khovanskii, Newton-Okounkov bodies, semigroups of integral points, graded algebras and intersection theory,  Annals of Math. 176 (2012), 925 - 978.

\bibitem{LM} R. Lazarsfeld and M. Musta\c{t}\u{a}, Convex bodies associated to linear series, Ann. Sci. Ec. Norm. Super 42 (2009) 783 - 835.

	
	
\bibitem{Li} J. Lipman, Equimultiplicity, reduction and blowing up,  R.N. Draper (Ed.), Commutative Algebra, Lect. Notes Pure Appl. Math., vol. 68, Marcel Dekker, New York (1982), pp. 111-147.
		
	



\bibitem{Mus} M. Musta\c{t}\u{a}, On multiplicities of graded sequences of ideals, J. Algebra 256 (2002), 229-249.





\bibitem{Ok} A. Okounkov, Why would multiplicities be log-concave?, in The orbit method in geometry and physics, Progr. Math. 213, 2003, 329-347.

\bibitem{Ra} J.L. Ratliff Jr., Locally quasi-unmixed Noetherian rings and ideals of the principal class, Pacific J. Math. 52 (1974), 185-205.

\bibitem{R} D. Rees, $\mathcal A$-transforms of local rings and a theorem on multiplicities of ideals, Proc. Cambridge Philos. Soc. 57 (1961), 8 - 17.

\bibitem{R1} D. Rees, Multiplicities, Hilbert functions and degree functions. In Commutative algebra: Durham 1981
(Durham 1981), London Math. Soc. Lecture Note Ser. 72, Cambridge, New York, Cambridge Univ. Press, 1982, 170 - 178.

\bibitem{R3} D. Rees, Izumi's theorem, in Commutative Algebra,  C. Huneke and J.D. Sally editors, Springer-Verlag 1989, 407 - 416.
\bibitem{RS} D. Rees and R. Sharp, On a Theorem of B. Teissier on Multiplicities of Ideals in Local Rings, J. London Math. Soc. 18 (1978), 449-463.

\bibitem{Ro} P. C.  Roberts, A prime ideal in a polynomial ring whose symbolic blow-up is not Noetherian'', Proc. AMS 1985, 589-592.

\bibitem{Se} J-P Serre, Alg\`ebre Locale, Multipliciti\'es, Springer Verlag, Lecture Notes in Mathematics 11, 1965.

\bibitem{S} I. Swanson, Mixed multiplicities, joint reductions and a theorem of Rees, J. London Math. Soc. 48 (1993), 1 - 14.


\bibitem{HS} I. Swanson and C. Huneke, Integral Closure of Ideals, Rings and Modules, Cambridge University Press, 2006.

\bibitem{T1} B. Teissier, Cycles \'evanescents, sections planes et conditions de Whitney, Singularit\'es \`a Carg\`ese 1972, Ast\'erisque 7-8 (1973)

\bibitem{T2} B. Teissier, Sur une in\'egalit\'e \`a la Minkowski pour les multiplicit\'es (Appendix to a paper by D. Eisenbud and H. Levine),
Ann. Math. 106 (1977), 38 - 44.

\bibitem{T3} B. Teissier, On a Minkowski type  inequality for multiplicities II, In C.P. Ramanujam - a tribute, Tata Inst. Fund. Res. Studies in Math. 8, Berlin - New York, Springer, 1978.

\bibitem{TV} N.V. Trung and J. Verma, Mixed multiplicities of ideals versus mixed volumes of polytopes, Trans. Amer. Math. Soc. 359 (2007), 4711 - 4727.

\end{thebibliography}
\end{document}